\documentclass[10pt]{amsart}

\newcounter{mnote}
\let\oldmarginpar\marginpar
\renewcommand\marginpar[1]{\-\oldmarginpar[\raggedleft\footnotesize #1]%
{\raggedright\footnotesize #1}}

\usepackage[english]{babel}

\usepackage{amssymb}
\usepackage{amsmath}
\usepackage{mathrsfs}
\usepackage{stmaryrd}
\usepackage{chemarrow}
\usepackage[norelsize]{algorithm2e}
\usepackage{enumerate}
\usepackage{graphicx}
\usepackage[all]{xy}
\usepackage{tikz,tikz-cd}
\usetikzlibrary{arrows}
\usepackage{multirow}
\usepackage[colorinlistoftodos]{todonotes}
\usepackage[colorlinks=true, allcolors=blue]{hyperref}
\usepackage[hyperpageref]{backref}
\usepackage{adjustbox}

\usepackage{bm}         

\usepackage{subcaption} 
\usepackage{float}
\usepackage{subcaption}
\newtheorem{theorem}{Theorem}[section]
\newtheorem{lemma}[theorem]{Lemma}

\newtheorem{example}[theorem]{Example}

\newtheorem{remark}[theorem]{Remark}

\newcommand{\dx}{\,{\rm d}x}

\newcommand{\curl}{\operatorname{curl}}
\renewcommand{\div}{\operatorname{div}}
\newcommand{\grad}{\operatorname{grad}}

\newcommand{\dev}{\operatorname{dev}}
\newcommand{\sym}{\operatorname{sym}}
\newcommand{\skw}{\operatorname{skw}}

\newcommand{\vskw}{\operatorname{vskw}}

\begin{document}
\title[Nonconforming Linear FEM for Tensor-Valued Stokes]{Nonconforming Linear Element Method for a Generalized Tensor-Valued Stokes Equation with Application to the Triharmonic Equation}
\author{Ziwen Gu}%
\address{School of Mathematics, Shanghai University of Finance and Economics, Shanghai 200433, China}%
\email{2024310134@stu.sufe.edu.cn}%
\author{Xuehai Huang}%
\address{School of Mathematics, Shanghai University of Finance and Economics, Shanghai 200433, China}%
\email{huang.xuehai@sufe.edu.cn}%



\begin{abstract}
A nonconforming linear element method is developed for a three-dimensional generalized tensor-valued Stokes equation associated with the Hessian complex in this paper. A discrete Helmholtz decomposition for the piecewise constant space of traceless tensors is established, ensuring the well-posedness of the nonconforming method, and optimal error estimates are derived. Building on this, a low-order decoupled finite element method for the three-dimensional triharmonic equation is constructed by combining the Morley-Wang-Xu element methods for the biharmonic subproblems with the proposed nonconforming linear element method. Numerical experiments confirm the theoretical convergence rates.
\end{abstract}

\maketitle


\section{Introduction}
\label{intro}
Let $ \Omega \subset \mathbb{R}^3 $ be a bounded polyhedral domain. In this paper, we develop a low-order finite element method for the following generalized tensor-valued Stokes equation with a given right-hand side $ \boldsymbol{g} \in L^2(\Omega; \mathbb{S})$: Find $\boldsymbol{\sigma} \in H_0^1(\Omega; \mathbb{S})$, $\boldsymbol{p}\in H_0(\div,\Omega; \mathbb{T})$, and $\boldsymbol{r}\in(H^1(\Omega; \mathbb{R}^3) / \textrm{RT})$ such that
\begin{equation}\label{generalized_stokes}
\left\{
\begin{aligned}
-\Delta\boldsymbol{\sigma}+\sym\curl\boldsymbol{p} = &\, \boldsymbol{g}     && \mathrm{in}\,\,\Omega,\\
\curl \boldsymbol{\sigma} + \dev\grad \boldsymbol{r} = &\, 0  && \mathrm{in}\,\, \Omega,\\
\div \boldsymbol{p} = &\, 0  && \mathrm{in}\,\, \Omega,
\end{aligned}\right.
\end{equation}
where $H_0^1(\Omega; \mathbb{S}) := H_0^1(\Omega) \otimes \mathbb{S}$, and $H^1(\Omega; \mathbb{R}^3) / \mathrm{RT}$ denotes the subspace of $H^1(\Omega; \mathbb{R}^3)$ orthogonal to $\textrm{RT} := \{a \boldsymbol{x} + \boldsymbol{b} : a \in \mathbb{R}, \boldsymbol{b}\in \mathbb{R}^3\} $ under the $L^2$ inner product. The divergence-related spaces are defined as
\begin{align*}
H(\div,\Omega; \mathbb{T}) & := \{\boldsymbol{\tau} \in L^{2}(\Omega; \mathbb{T} )
: \div \boldsymbol{\tau} \in L^{2}(\Omega; \mathbb{R}^3)\},\\
H_{0}(\div,\Omega; \mathbb{T} ) & := \{\boldsymbol{\tau} \in H(\div,\Omega; \mathbb{T})
: \boldsymbol{\tau} \boldsymbol{n} = 0 \,\, \mathrm{on}\,\, \partial\Omega\}.
\end{align*}
Here, $\mathbb{S}$ and $\mathbb{T}$ are subspaces of $\mathbb{M} := \mathbb{R}^{3\times3}$ consisting of symmetric and traceless matrices, respectively.
The first equation in \eqref{generalized_stokes} is understood in $H^{-1}(\Omega;\mathbb S)$. More precisely, for $\boldsymbol{\tau}\in H_0^1(\Omega;\mathbb S)$,
\begin{equation*}
\left\langle -\Delta\boldsymbol{\sigma}+\sym\curl\boldsymbol{p},\boldsymbol{\tau}\right\rangle
:=(\nabla\boldsymbol{\sigma},\nabla\boldsymbol{\tau})
+(\boldsymbol{p},\curl\boldsymbol{\tau}).
\end{equation*}
Thus, the operators $-\Delta$ and $\sym\curl$ in the first equation are not assumed to be strong differential operators.
The generalized tensor-valued Stokes equation \eqref{generalized_stokes} arises in applications such as the triharmonic equation \cite{ChenHuang2018,Gallistl2017}.

The well-posedness of problem \eqref{generalized_stokes} is closely connected to the following Hessian complex:
\begin{equation}\label{intro:hessiancomplex}
	0\xrightarrow{\subset} H_0^3(\Omega)\xrightarrow{\nabla^2} H_0^1(\Omega;\mathbb S)\xrightarrow{\curl} H_0(\div,\Omega;\mathbb T)\xrightarrow{\div}L^2(\Omega; \mathbb{R}^3) / \textnormal{RT} \xrightarrow{}0.
\end{equation}
This complex is smoother than the domain complex of the Hessian complex:
\begin{equation}\label{intro:hessiancomplexH2}
	0\xrightarrow{\subset} H_0^2(\Omega)\xrightarrow{\nabla^2} H_0(\curl,\Omega;\mathbb S)\xrightarrow{\curl} H_0(\div,\Omega;\mathbb T)\xrightarrow{\div}L^2(\Omega; \mathbb{R}^3) / \textnormal{RT} \xrightarrow{}0,
\end{equation}
which has applications in solving the biharmonic equation~\cite{ChenHuang2025a,ChenHuHuang2018,PaulyZulehner2020} and the Einstein-Bianchi equation~\cite{QuennevilleBelair2015}.
Conforming discretizations of these Hessian complexes \eqref{intro:hessiancomplex}-\eqref{intro:hessiancomplexH2} have been recently developed in~\cite{ChenHuang2025,ChenHuang2022,HuLiang2021,HuLiangLin2024}. However, such discretizations require high-order polynomial spaces and supersmooth degrees of freedom (DoFs), which limits their practical efficiency for solving the generalized tensor-valued Stokes equation \eqref{generalized_stokes}.

To date, numerical studies for problem \eqref{generalized_stokes} are scarce. The only available approaches, to the best of our knowledge, are the low-order mixed finite element methods proposed in~\cite[Section~5.3]{Gallistl2017}, which involve two unknowns, $\boldsymbol{\sigma}$ and $\boldsymbol{p}$. In that work, $\boldsymbol{p}$ is discretized using
\begin{equation*}
	P_h = \{\boldsymbol{p} \in H_0(\div, \Omega; \mathbb{T}) : \boldsymbol{p}|_T \text{ is constant on each } T \in \mathcal{T}_h \}.
\end{equation*}
However, as the space $P_h$ lacks local DoFs, its precise implementation remains unclear and, as noted in~\cite{ChenHuang2022,ChenHuang2024a}, may rely on vertex-associated DoFs.

To circumvent the vertex DoFs of $H(\div)$-conforming finite elements for traceless tensors, we adopt the
following weak formulation of problem \eqref{generalized_stokes}: Find $(\boldsymbol{\sigma}, \boldsymbol{p}, \boldsymbol{r})  \in H_0^1(\Omega; \mathbb{S}) \times L^2(\Omega; \mathbb{T}) \times (H^1(\Omega; \mathbb{R}^3) / \textrm{RT})$ such that
\begin{subequations}\label{intro:continue_generalized_stokes}
\begin{align}
( \nabla \boldsymbol{\sigma}, \nabla \boldsymbol{\tau}) + ( \curl \boldsymbol{\tau} + \dev\grad \boldsymbol{s}, \boldsymbol{p}) & = ( \boldsymbol{g}, \boldsymbol{\tau} ), && \label{intro:continue_generalized_stokes1} \\
( \curl \boldsymbol{\sigma} + \dev\grad \boldsymbol{r}, \boldsymbol{q}) & = 0,  &&  \label{intro:continue_generalized_stokes2}
\end{align}
for all $(\boldsymbol{\tau}, \boldsymbol{q}, \boldsymbol{s})  \in H_0^1(\Omega; \mathbb{S}) \times L^2(\Omega; \mathbb{T}) \times (H^1(\Omega; \mathbb{R}^3) / \textrm{RT})$.
\end{subequations}
In other words, $\boldsymbol{p}$ is taken in $L^2(\Omega; \mathbb{T})$ instead of $H_0(\div, \Omega; \mathbb{T})$, thereby avoiding vertex DoFs.
The well-posedness of the weak formulation \eqref{intro:continue_generalized_stokes} is established via the following Helmholtz decomposition for traceless tensors:
\begin{equation}\label{intro:helmholtzdecomposition}
L^2(\Omega; \mathbb{T} ) = \curl H_0^1(\Omega; \mathbb{S}) \oplus \dev\grad( H^1(\Omega; \mathbb{R}^3)/\textnormal{RT} ),
\end{equation}
which can equivalently be expressed as the short complex:
\begin{equation*}
H_{0}^{1}(\Omega;\mathbb S)\times H^{1}(\Omega;\mathbb R^3) \xrightarrow{(\curl, \dev\grad)} L^2(\Omega; \mathbb{T}) \to 0.
\end{equation*}

We discretize $\boldsymbol{\sigma}$ and $\boldsymbol{r}$ using nonconforming linear elements and $\boldsymbol{p}$ using piecewise constants: Find
$(\boldsymbol{\sigma}_h, \boldsymbol{p}_h, \boldsymbol{r}_h) \in  \mathring{V}_h^{\mathbb{S}} \times \mathbb{P}_0(\mathcal{T}_h; \mathbb{T}) \times  V_h $ such that
\begin{subequations}\label{intro:discrete_generalized_stokes}
\begin{align}
a_h(\boldsymbol{\sigma}_h, \boldsymbol{r}_h; \boldsymbol{\tau}, \boldsymbol{s}) + b_h(\boldsymbol{\tau}, \boldsymbol{s}; \boldsymbol{p}_h ) & = ( \boldsymbol{g}_h, \boldsymbol{\tau} ), && \forall\, \boldsymbol{\tau} \in \mathring{V}_h^{\mathbb{S}}, \boldsymbol{s} \in V_h , \label{intro:discrete_generalized_stokes_1}        \\
b_h(\boldsymbol{\sigma}_h, \boldsymbol{r}_h; \boldsymbol{q} ) & = 0, && \forall\, \boldsymbol{q} \in \mathbb{P}_0(\mathcal{T}_h; \mathbb{T}), \label{intro:discrete_generalized_stokes_2}
\end{align}
\end{subequations}
where $\boldsymbol{g}_h\in L^2(\Omega; \mathbb{S})$ is an approximation of $\boldsymbol{g}$, and the discrete bilinear forms
\begin{align*}
a_h(\boldsymbol{\sigma}, \boldsymbol{r}; \boldsymbol{\tau}, \boldsymbol{s}) & := ( \nabla_h \boldsymbol{\sigma}, \nabla_h \boldsymbol{\tau} ) + \sum_{F \in \Delta_{2}(\mathcal{T}_h) } h_F^{-1} ( [\boldsymbol{r}], [\boldsymbol{s}] )_{F},  \\
b_h(\boldsymbol{\tau}, \boldsymbol{s}; \boldsymbol{p} ) & := ( \curl_h \boldsymbol{\tau} + \dev\grad _h \boldsymbol{s}, \boldsymbol{p}).
\end{align*}
Here, $\mathring{V}_h^{\mathbb{S}}$ is the nonconforming linear element space with homogeneous boundary condition for symmetric tensors, $V_h$ is the nonconforming linear element space for vectors, and $\mathbb{P}_0(\mathcal{T}_h; \mathbb{T})$ is the piecewise constant space for traceless tensors.
The term $\sum_{F \in \Delta_{2}(\mathcal{T}_h) } h_F^{-1} \left( [\boldsymbol{r}], [\boldsymbol{s}] \right)_{F}$ is incorporated into the bilinear form $a_h(\cdot, \cdot; \cdot, \cdot)$ to guarantee its discrete coercivity and the uniqueness of $\boldsymbol{r}_h$, as ensured by the broken Korn-type inequality \eqref{discrete_korn_rt}.
We refer to \cite{AnHuangZhang2024,Gallistl2017,Schedensack2016} for some finite element methods for generalized tensor-valued Stokes equations in two dimensions.

Analogously to the Helmholtz decomposition \eqref{intro:helmholtzdecomposition}, we establish the discrete Helmholtz decomposition
\begin{equation*}
\mathbb{P}_0(\mathcal{T}_h; \mathbb{T}) = \textup{curl}_h \mathring{V}_h^{\mathbb{S}} + \dev\grad_h V_h,
\end{equation*}
ensuring the well-posedness of the nonconforming linear element method \eqref{intro:discrete_generalized_stokes}. Optimal error estimates are derived for this nonconforming linear element method.

We further employ this nonconforming linear element method to construct a low-order decoupled finite element method for the three-dimensional triharmonic equation with $f \in L^2(\Omega)$: Find $u\in H_0^3(\Omega)$ such that
\begin{equation}\label{triharmonic}
-\Delta^{3}u = f  \quad \mathrm{in}\,\,\Omega.
\end{equation}
This sixth-order elliptic partial differential equation arises in numerous applications, including gradient-elastic Kirchhoff plate problems \cite{ChenHuangHuang2024,KP2008}, high-order phase-field models \cite{PawlowZajaczkowski2011,SchimpernaPawlow2013}, and thin film problems \cite{TF1}.
Conforming discretizations require $C^2$-continuous polynomial spaces~\cite{HuLinWu2024,ChenChenGaoHuangEtAl2025,ChenHuang2021,ChenHuang2024}, which suffer from supersmooth DoFs and higher polynomial degrees. We refer to \cite{WangXu2013,LiWu2025,JinWu2025} for nonconforming finite element methods, \cite{DroniouIlyasLamichhaneWheeler2019,DassiMoraRealesVelasquez2024} for mixed methods, \cite{ChenHuang2020,Huang2020,ChenHuangWei2022} for virtual element methods and \cite{GudiNeilan2011,DG2017,ChenLiQiu2022} for discontinuous Galerkin methods.

Following the methodology from \cite{ChenHuang2018}, the triharmonic equation \eqref{triharmonic} is decomposed into two biharmonic equations and one generalized tensor-valued Stokes equation.
The decoupled formulation facilitates the development of efficient finite element methods and the design of fast solvers.
Specifically, we discretize the two biharmonic equations using the Morley-Wang-Xu element~\cite{morley2006}, and approximate the generalized tensor-valued Stokes equation using the nonconforming linear element method~\eqref{intro:discrete_generalized_stokes}. This results in a low-order decoupled finite element method for the three-dimensional triharmonic equation, for which we derive optimal error estimates. For alternative decoupled methods for solving the triharmonic equation, we refer to \cite{Gallistl2017,AnHuangZhang2024,Schedensack2016,C0triharmonic}.

The rest of the paper is organized as follows. The well-posedness of the generalized tensor-valued Stokes equation is established in Section \ref{section2}.
Section \ref{section3} focuses on the nonconforming linear element method for the generalized tensor-valued Stokes equation and its error analysis. Section \ref{section4} presents a low-order decoupled method for the triharmonic equation. Numerical experiments are presented in Section \ref{section5} to validate the theoretical results.

\section{Generalized Tensor-Valued Stokes Equation}\label{section2}

In this section, we introduce the notation used throughout the paper, and establish the well-posedness of the weak formulation associated with the generalized tensor-valued Stokes equation~\eqref{generalized_stokes}.

\subsection{Notation}
Denote by $\mathbb{M}$ the space of all $3 \times 3$ matrices, by $\mathbb{S}$ the subspace of symmetric matrices, and by $\mathbb{T}$ the subspace of traceless matrices. For a matrix $\boldsymbol{\tau} \in \mathbb{M} $, we can decompose it into the traceless part and the diagonal part:
\begin{equation*}
\boldsymbol{\tau} = \dev \boldsymbol{\tau} + \frac{1}{3}\text{tr}(\boldsymbol{\tau})\boldsymbol{I} := (\boldsymbol{\tau} - \frac{1}{3}\text{tr}(\boldsymbol{\tau})\boldsymbol{I}) + \frac{1}{3}\text{tr}(\boldsymbol{\tau})\boldsymbol{I}.
\end{equation*}
Denote by $\sym\boldsymbol{\tau}$ the symmetric part of $\boldsymbol{\tau}$, i.e., $\sym\boldsymbol{\tau} = (\boldsymbol{\tau} + \boldsymbol{\tau}^{\intercal})/2$. We also denote the skew-symmetric part by
$\skw\boldsymbol{\tau}:=(\boldsymbol{\tau}-\boldsymbol{\tau}^{\intercal})/2$. For $\boldsymbol{\tau}=(\tau_{ij})_{i,j=1}^3\in\mathbb M$, define
\[
\vskw\boldsymbol{\tau}:=\frac12(\tau_{32}-\tau_{23},\,\tau_{13}-\tau_{31},\,\tau_{21}-\tau_{12})^{\intercal},
\]
so that $\vskw$ is the axial-vector map associated with the skew-symmetric part of $\boldsymbol{\tau}$.

Given an integer $m \geq 0$ and a bounded domain $D \subset \mathbb{R}^{3}$, we define $H^{m}(D)$  as the standard Sobolev space of functions on $D$. The corresponding norm and seminorm are denoted by $\| \cdot \|_{ m,D}$ and $|\cdot |_{m,D}$, respectively. We abbreviate $\|\cdot\|_{0,D}$ as $\|\cdot\|_{D}$. Set $L^{2}(D) = H^{0}(D)$.  Let $L_{0}^{2}(D)$ be the space of functions in $L^{2}(D)$ with vanishing integral average values. For a space $B(D)$ defined on $D$, let $B(D;\mathbb{X}):=B(D) \otimes \mathbb{X}$ be its vector or tensor version, where the tensor space $ \mathbb{X} $ can be taken as $\mathbb{R}^{3}, \mathbb{M}, \mathbb{S}, \mathbb{T} $, etc. We denote $(\cdot, \cdot)_{D}$ as the usual inner product on $L^{2}(D)$ or $L^{2}(D; \mathbb{X})$. We denote $H_{0}^{m}(D)\,(H_{0}^{m}(D; \mathbb{X}))$ as the closure of $C_{0}^{\infty}(D)\,(C_{0}^{\infty}(D; \mathbb{X}))$ with respect to the norm $\| \cdot \|_{m,D}$. When $D$ is $\Omega$, we abbreviate $\|\cdot\|_{D}$, $\| \cdot \|_{ m,D}$, $|\cdot|_{m,D}$ and $(\cdot, \cdot)_{D}$  as $\|\cdot\|$, $\| \cdot \|_{m}$, $|\cdot|_{m}$ and $(\cdot, \cdot)$, respectively.

We use $\boldsymbol{n}_{\partial D}$ to denote the unit outward normal vector of $\partial D$, which will be abbreviated as $\boldsymbol{n}$ if it does not cause any confusion. Let $\{\mathcal{T}_{h}\}_{h>0}$ be a regular family of tetrahedral meshes of $\Omega$, where $h=\max_{T \in \mathcal{T}_{h}}h_T$ with $h_T$ being the diameter of tetrahedron $T$. For $\ell=0,1,2$, denote by $\Delta_{\ell}(\mathcal T_h)$ and $\Delta_{\ell}(\mathring{\mathcal T}_h)$ the set of all subsimplices and all interior subsimplices of dimension $\ell$ in the partition $\mathcal{T}_{h}$, respectively.
For a face $F\in \Delta_{2}(\mathcal T_h)$, we choose one of the two unit normal vectors to $F$ arbitrarily and keep this choice fixed; the chosen vector is denoted by $\boldsymbol{n}_F$. All jumps below are taken with respect to this fixed orientation.
For two adjacent tetrahedra $T_1$ and $T_2$ sharing an interior face $F$, we define the jump of a function $w$ on $F$ as
\begin{equation*}
[w]:=(w|_{T_1})|_F\boldsymbol{n}_F\cdot\boldsymbol{n}_{\partial T_1}+(w|_{T_2})|_F\boldsymbol{n}_F\cdot\boldsymbol{n}_{\partial T_2}.
\end{equation*}
On a face $F$ lying on the boundary $\partial\Omega$, the jump becomes $[w]:=w|_F$.

We denote the gradient operator, curl operator and divergence operator as $\grad\,(\nabla)$, $\curl$ and $\text{div}$, and let $\grad_{h}\, (\nabla_h)$, $\curl_{h}$  and $\text{div}_{h}$  be the element-wise counterpart of $\grad$, $\curl$ and $\text{div}$ with respect to $\mathcal{T}_{h}$. For a tensor-valued function, these operators are applied row-wise.
We introduce the Sobolev spaces
\begin{align*}
H(\div\div, D; \mathbb{S}) & := \{\boldsymbol{\tau} \in L^{2}(D; \mathbb{S})
: \div\div \boldsymbol{\tau} \in L^{2}(D) \},\\
H(\sym\curl, D; \mathbb{T}) & := \{\boldsymbol{\tau} \in L^{2}(D; \mathbb{T})
: \sym\curl \boldsymbol{\tau} \in L^{2}(D; \mathbb{S} )\},\\
H(\text{div}, D; \mathbb{T}) & := \{\boldsymbol{\tau} \in L^{2}(D; \mathbb{T} )
: \div \boldsymbol{\tau} \in L^{2}(D; \mathbb{R}^3)\},\\
H_{0}(\text{div}, D; \mathbb{T} ) & := \{\boldsymbol{\tau} \in H(\text{div},D; \mathbb{T})
: \boldsymbol{\tau} \boldsymbol{n} = 0 \,\, \mathrm{on}\,\, \partial D\}.
\end{align*}
For an integer $k\geq0$, let $\mathbb{P}_{k}(D)$ denote the space of all polynomials in $D$ of total degree at most $k$, and introduce the piecewise smooth spaces
\[
H^1(\mathcal{T}_h):= \{ v \in L^2(\Omega) : v|_T \in H^1(T) \ \text{for all } T \in \mathcal{T}_h \},
\]
\[
\mathbb{P}_k(\mathcal{T}_h) := \{ v \in L^2(\Omega) : v|_T \in \mathbb{P}_k(T) \ \text{for all } T \in \mathcal{T}_h \}.
\]
For a piecewise smooth function $v$, define the following broken seminorms:
\begin{equation*}
|v|_{s,h}^2 :=  \sum_{T\in\mathcal{T}_h} | v |_{s,T}^2,\quad
\interleave v\interleave_{1,h}^2 :=  |v|_{1,h}^2+\sum_{F \in\Delta_{2}(\mathcal T_h)}  h_F^{-1}\|[v]\|_{F}^2.
\end{equation*}
By the broken Poincar\'e inequality in \cite[(1.8)]{Brenner2003}, $\interleave v\interleave_{1,h}$ is a norm on the space $H^1(\mathcal{T}_h)$.

In this paper, we use ``$\lesssim \cdots$'' to mean that ``$\leq C \cdots$'', where $C$ is a generic positive constant independent of $h$, which may take different values in different contexts. Moreover, $A \eqsim B$ means that $A \lesssim B$ and $B \lesssim A$.

\subsection{Hessian complex and Helmholtz decomposition}
We start by presenting a Hessian complex and a Helmholtz decomposition for the traceless tensor space $L^2(\Omega; \mathbb{T})$.
\begin{lemma}
For a contractible domain $\Omega$, the following exact sequence holds:
\begin{equation}\label{hessiancomplex}
0\xrightarrow{\subset} H_0^3(\Omega)\xrightarrow{\nabla^2} H_0^1(\Omega;\mathbb S)\xrightarrow{\curl} H_0(\div,\Omega;\mathbb T)\xrightarrow{\div}L^2(\Omega; \mathbb{R}^3) / \textnormal{RT} \xrightarrow{}0.
\end{equation}

\end{lemma}
\begin{proof}
First, we prove
\[
H_0^1(\Omega;\mathbb S)\cap\ker(\curl)=\nabla^2 H_0^3(\Omega).
\]
The inclusion $\nabla^2 H_0^3(\Omega)\subset H_0^1(\Omega;\mathbb S)\cap\ker(\curl)$ is obvious. Conversely, let $\boldsymbol{\tau}\in H_0^1(\Omega;\mathbb S)$ satisfy $\curl\boldsymbol{\tau}=0$. Since $\Omega$ is contractible, the relevant cohomology spaces in Theorem 1.1 of \cite{CostabelMcIntosh2010} are trivial. Applying this theorem row by row to the de Rham complex with homogeneous boundary condition, there exists $\boldsymbol{v}\in H_0^2(\Omega;\mathbb R^3)$ such that $\boldsymbol{\tau}=\grad\boldsymbol{v}$. Since $\boldsymbol{\tau}$ takes values in $\mathbb S$, we have $\skw(\grad\boldsymbol{v})=0$. By the identity $\curl\boldsymbol{v}=2\vskw(\grad\boldsymbol{v})$, cf. Fig.1 in \cite{ChenHuang2025}, this implies $\curl\boldsymbol{v}=0$. Applying Theorem 1.1 in \cite{CostabelMcIntosh2010} again gives a function $u\in H_0^3(\Omega)$ such that $\boldsymbol{v}=\grad u$. Hence $\boldsymbol{\tau}=\nabla^2 u$. We also refer to Theorem 3.1 in \cite{ChenHuang2022} for a similar proof of the Hessian complex.
Moreover, $\div H_0(\div,\Omega;\mathbb T)=L^2(\Omega; \mathbb{R}^3) / \textnormal{RT}$ was established in \cite[Theorem 3.12]{PaulyZulehner2020}. Finally, by Proposition 13 in \cite{Gallistl2017}, one has
\begin{equation}\label{eq:curlH1S}
H_0(\div, \Omega; \mathbb{T}) \cap \ker(\div) =\curl H_0^1(\Omega; \mathbb{S}).
\end{equation}
Hence, the sequence \eqref{hessiancomplex} is exact.
\end{proof}

Recall the Korn-type inequality for the traceless gradient operator $\dev\grad$ \cite[Lemma 3.2]{PaulyZulehner2020}:
\begin{equation}\label{Korninequality}
\| \boldsymbol{v} \|_{1} \eqsim \|\dev\grad \boldsymbol{v} \|, \quad \forall\,\boldsymbol{v} \in H^1(\Omega; \mathbb{R}^3)/ \textrm{RT}.
\end{equation}

\begin{lemma}\label{lem:helm_decomposition}
The following Helmholtz decomposition holds:
\begin{equation}\label{continue_helm}
L^2(\Omega; \mathbb{T}) = \curl H_0^1(\Omega; \mathbb{S}) \oplus \dev\grad( H^1(\Omega; \mathbb{R}^3)/\textnormal{RT} ).
\end{equation}
That is, for any $ \boldsymbol{q} \in L^2(\Omega; \mathbb{T}) $, there exist $ \boldsymbol{\tau} \in H_0^1(\Omega; \mathbb{S}) $ and $ \boldsymbol{s} \in H^1(\Omega; \mathbb{R}^3)/ \textnormal{RT} $ such that
\begin{equation*}
\boldsymbol{q} = \curl\boldsymbol{\tau} + \dev\grad \boldsymbol{s}, \quad \|\boldsymbol{\tau}\|_1 + \|\boldsymbol{s}\|_1 \lesssim \|\boldsymbol{q}\|.
\end{equation*}
\end{lemma}
\begin{proof}
It is clear that the sum is direct under the $L^2$ inner product, and
\begin{equation*}
\curl H_0^1(\Omega; \mathbb{S}) \oplus \dev\grad ( H^1(\Omega; \mathbb{R}^3)/\textrm{RT} )  \subseteq L^2(\Omega; \mathbb{T}).
\end{equation*}
Now we prove the reverse inclusion. Take any $ \boldsymbol{q} \in L^2(\Omega; \mathbb{T})$.
Consider the following variational problem: Find $ \boldsymbol{s} \in H^1(\Omega; \mathbb{R}^3)/\textrm{RT} $ such that
\begin{equation*}
( \dev\grad \boldsymbol{s}, \dev\grad \boldsymbol{t}  ) = ( \boldsymbol{q}, \dev\grad \boldsymbol{t}), \quad \forall\,\boldsymbol{t} \in H^1(\Omega; \mathbb{R}^3)/\textrm{RT}.
\end{equation*}
The well-posedness of this problem is guaranteed by the inequality \eqref{Korninequality} and the Lax-Milgram lemma \cite{Ciarlet1978}. We also have
\begin{equation*}
\|\boldsymbol{s}\|_1\lesssim \|\boldsymbol{q}\|.
\end{equation*}
Set $\boldsymbol{\eta}:=\boldsymbol{q}-\dev\grad\boldsymbol{s}$. Then
\[
(\boldsymbol{\eta},\dev\grad\boldsymbol{t})=0
\qquad
\forall\,\boldsymbol{t}\in H^1(\Omega;\mathbb R^3)/\textnormal{RT}.
\]
We use the characterization
\begin{equation*}
\bigl(\dev\grad(H^1(\Omega;\mathbb R^3)/\textnormal{RT})\bigr)^\perp
= H_0(\div,\Omega;\mathbb T)\cap\ker(\div),
\end{equation*}
where the orthogonal complement is taken in $L^2(\Omega;\mathbb T)$. Indeed, since $\boldsymbol{\eta}$ is trace-free,
$(\boldsymbol{\eta},\dev\grad\boldsymbol{t})=(\boldsymbol{\eta},\grad\boldsymbol{t})$.
The Green formula
\[
(\boldsymbol{\eta},\grad\boldsymbol{t})
=-(\div\boldsymbol{\eta},\boldsymbol{t})
+ (\boldsymbol{\eta}\boldsymbol{n},\boldsymbol{t})_{\partial\Omega}
\]
shows that the orthogonality is equivalent to $\div\boldsymbol{\eta}=0$ in $\Omega$ and $\boldsymbol{\eta}\boldsymbol{n}=0$ on $\partial\Omega$. Hence
$\boldsymbol{q}-\dev\grad\boldsymbol{s}\in H_0(\div,\Omega;\mathbb T)\cap\ker(\div)$.
Finally, the Helmholtz decomposition~\eqref{continue_helm} follows from \eqref{eq:curlH1S}.
\end{proof}

The Helmholtz decomposition \eqref{continue_helm} yields the following complex.
\begin{lemma}
For the contractible domain $\Omega$, the complex
\begin{equation}\label{curldevgradcomplex}
H_0^3(\Omega)\times \textnormal{RT}\xrightarrow{\begin{pmatrix}\nabla^2 & 0\\
0&\subset\end{pmatrix}}H_{0}^{1}(\Omega;\mathbb S)\times H^{1}(\Omega;\mathbb R^3) \xrightarrow{(\curl, \dev\grad)} L^2(\Omega; \mathbb{T}) \to 0
\end{equation}
is exact.
\end{lemma}
\begin{proof}
The Helmholtz decomposition \eqref{continue_helm} implies the operator $ (\curl, \dev\grad): H_{0}^{1}(\Omega;\mathbb S)\times H^{1}(\Omega;\mathbb R^3)\to L^2(\Omega; \mathbb{T})$ is surjective.
The exactness of complex \eqref{curldevgradcomplex} then follows from the Hessian complex \eqref{hessiancomplex} and $\ker(\dev\grad)=\textrm{RT}$.
\end{proof}

The surjectivity of the operator $(\curl,\dev\grad)$ in the complex \eqref{curldevgradcomplex} can also be derived by constructing a commutative diagram based on the first part of the following complex:
\begin{equation*}
\textrm{RT}\xrightarrow{\subset} H^1(\Omega;\mathbb R^3)\xrightarrow{\dev\grad} L^2(\Omega;\mathbb T)\xrightarrow{\sym\curl} H^{-1}(\Omega;\mathbb S)\xrightarrow{\div\div}H^{-3}(\Omega) \xrightarrow{}0.
\end{equation*}
We refer to \cite{ChenHuang2018} for further details on this approach.

Thanks to the complex \eqref{curldevgradcomplex}, the operator $(\curl, \dev\grad)$ can be interpreted as a generalized divergence operator, and problem \eqref{generalized_stokes} as a generalized tensor-valued Stokes equation involving tensor-valued unknowns.

\subsection{Weak formulation of generalized tensor-valued Stokes equation}
A weak formulation of the generalized tensor-valued Stokes equation \eqref{generalized_stokes} is to find $( \boldsymbol{\sigma}, \boldsymbol{p}, \boldsymbol{r}  )  \in H_0^1(\Omega; \mathbb{S}) \times L^2(\Omega; \mathbb{T}) \times (H^1(\Omega; \mathbb{R}^3) / \textrm{RT})  $ such that
\begin{subequations}\label{continue_generalized_stokes}
\begin{align}
a(\boldsymbol{\sigma} , \boldsymbol{r}; \boldsymbol{\tau} , \boldsymbol{s}) + b( \boldsymbol{\tau}, \boldsymbol{s}; \boldsymbol{p}) & = \left( \boldsymbol{g}, \boldsymbol{\tau} \right), && \forall\, \boldsymbol{\tau} \in H_0^1(\Omega; \mathbb{S}), \boldsymbol{s} \in H^1(\Omega; \mathbb{R}^3) / \textrm{RT}, \label{continue_generalized_stokes1} \\
b( \boldsymbol{\sigma}, \boldsymbol{r}; \boldsymbol{q}) & = 0, && \forall\,\boldsymbol{q} \in L^2(\Omega; \mathbb{T}), \label{continue_generalized_stokes2}
\end{align}
\end{subequations}
where the bilinear forms
\begin{equation*}
a(\boldsymbol{\sigma}, \boldsymbol{r}; \boldsymbol{\tau}, \boldsymbol{s}) = ( \nabla \boldsymbol{\sigma}, \nabla \boldsymbol{\tau}), \quad b(\boldsymbol{\tau}, \boldsymbol{s}; \boldsymbol{q}) = ( \curl \boldsymbol{\tau} + \dev\grad \boldsymbol{s}, \boldsymbol{q}).
\end{equation*}

Using the Helmholtz decomposition \eqref{continue_helm}, we now establish the well-posedness of the weak formulation \eqref{continue_generalized_stokes}.

\begin{lemma}
For $\left(\boldsymbol{\tau}, \boldsymbol{s} \right) \in H_0^1(\Omega; \mathbb{S}) \times (H^1(\Omega; \mathbb{R}^3) / \textnormal{RT})$ satisfying
\begin{equation}\label{condition}
b(\boldsymbol{\tau}, \boldsymbol{s}; \boldsymbol{q}) = 0, \quad\forall\,\boldsymbol{q} \in L^2(\Omega; \mathbb{T}),
\end{equation}
we have
\begin{equation}\label{continue_coercive}
\| \boldsymbol{\tau} \|_1^2 + \| \boldsymbol{s} \|_1^2 \lesssim a(\boldsymbol{\tau}, \boldsymbol{s}; \boldsymbol{\tau}, \boldsymbol{s}).
\end{equation}
\end{lemma}
\begin{proof}
The condition \eqref{condition} implies
$$
\curl \boldsymbol{\tau} + \dev\grad \boldsymbol{s} = 0.
$$
Combined with the Helmholtz decomposition \eqref{continue_helm}, this yields $\curl \boldsymbol{\tau} = 0$ and $\dev\grad \boldsymbol{s} = 0$, and hence $\boldsymbol{s} = 0$.
Finally, \eqref{continue_coercive} follows from the Poincar\'e inequality.
\end{proof}

\begin{lemma}
The following inf-sup condition holds:
\begin{equation}
\| \boldsymbol{q} \| \lesssim \sup\limits_{\boldsymbol{\tau} \in H_0^1(\Omega; \mathbb{S}), \boldsymbol{s} \in H^1(\Omega; \mathbb{R}^3)/\textnormal{RT} }  \frac{b(\boldsymbol{\tau}, \boldsymbol{s}; \boldsymbol{q})}{ \| \boldsymbol{\tau} \|_1 + \| \boldsymbol{s} \|_1}, \qquad\forall\,\boldsymbol{q} \in L^2(\Omega; \mathbb{T}). \label{continue_infsup}
\end{equation}
\end{lemma}
\begin{proof}
By the Helmholtz decomposition \eqref{continue_helm}, there exist $\boldsymbol{\tau} \in H_0^1(\Omega; \mathbb S)$ and $ \boldsymbol{s} \in H^1(\Omega; \mathbb{R}^3)/\textrm{RT}$ such that
\begin{equation*}
\boldsymbol{q} = \curl \boldsymbol{\tau} + \dev\grad \boldsymbol{s},  \quad \| \boldsymbol{\tau} \|_1 + \| \boldsymbol{s} \|_1 \lesssim \| \boldsymbol{q} \|.
\end{equation*}
Then $b(\boldsymbol{\tau}, \boldsymbol{s}; \boldsymbol{q}) = \| \boldsymbol{q} \|^2$, and the inequality \eqref{continue_infsup} immediately holds.
\end{proof}

\begin{theorem}\label{gStokes:continue_wellposed}
The weak formulation \eqref{continue_generalized_stokes} is well-posed, and it is equivalent to the generalized tensor-valued Stokes equation \eqref{generalized_stokes} in the distributional sense. Moreover,  $\boldsymbol{r} = 0$, $\curl\boldsymbol{\sigma} = 0$, $\div\boldsymbol{p} = 0$ and $\boldsymbol{p} \in H_0(\div,\Omega; \mathbb{T})$.
\end{theorem}
\begin{proof}
Thanks to the coercivity \eqref{continue_coercive} and the inf-sup condition \eqref{continue_infsup}, the Babu\v{s}ka-Brezzi theory \cite{BoffiBrezziFortin2013} implies the well-posedness of the weak formulation \eqref{continue_generalized_stokes}.

Combining equation \eqref{continue_generalized_stokes2} with the Helmholtz decomposition \eqref{continue_helm} yields $\boldsymbol{r} = 0$ and $\curl \boldsymbol{\sigma} = 0$. Choosing $\boldsymbol{\tau} = 0$ in \eqref{continue_generalized_stokes1} gives $\boldsymbol{p} \in H_0(\div,\Omega; \mathbb{T})$ with $\div \boldsymbol{p} = 0$. Finally, taking $\boldsymbol{s}=0$ in \eqref{continue_generalized_stokes1} and using the definition given after \eqref{generalized_stokes}, we obtain $-\Delta\boldsymbol{\sigma}+\sym\curl\boldsymbol{p}=\boldsymbol{g}$ in $H^{-1}(\Omega;\mathbb S)$. Hence, \eqref{continue_generalized_stokes} is equivalent to \eqref{generalized_stokes} in the distributional sense.
\end{proof}

\begin{remark}\rm
By applying integration by parts to $(\dev\grad \boldsymbol{s}, \boldsymbol{p})$ and $(\dev\grad \boldsymbol{r}, \boldsymbol{q})$ in \eqref{continue_generalized_stokes},
we have another weak formulation of the generalized tensor-valued Stokes equation \eqref{generalized_stokes} : Find $( \boldsymbol{\sigma}, \boldsymbol{p}, \boldsymbol{r}  )  \in H_0^1(\Omega; \mathbb{S}) \times H_0(\div,\Omega; \mathbb{T}) \times (L^2(\Omega; \mathbb{R}^3) / \textrm{RT})  $ such that
\begin{align*}
a(\boldsymbol{\sigma} , \boldsymbol{r}; \boldsymbol{\tau} , \boldsymbol{s}) + \tilde{b}( \boldsymbol{\tau}, \boldsymbol{s}; \boldsymbol{p}) & = \left( \boldsymbol{g}, \boldsymbol{\tau} \right), && \forall\, \boldsymbol{\tau} \in H_0^1(\Omega; \mathbb{S}), \boldsymbol{s} \in L^2(\Omega; \mathbb{R}^3) / \textrm{RT}, \\
\tilde{b}( \boldsymbol{\sigma}, \boldsymbol{r}; \boldsymbol{q}) & = 0, && \forall\,\boldsymbol{q} \in H_0(\div,\Omega; \mathbb{T}), 
\end{align*}
where the bilinear form
\begin{equation*}
\tilde{b}(\boldsymbol{\tau}, \boldsymbol{s}; \boldsymbol{q}) = (\curl \boldsymbol{\tau}, \boldsymbol{q}) - (\boldsymbol{s}, \div\boldsymbol{q}).
\end{equation*}
A similar formulation with two unknowns was studied in \cite[(5.4)]{Gallistl2017}.
The well-posedness of this weak formulation is related to the Hessian complex \eqref{hessiancomplex}.
\end{remark}

\section{Nonconforming Linear Element Method for Generalized Tensor-valued Stokes Equation}\label{section3}
In this section, we develop and analyze a nonconforming linear element method for the generalized tensor-valued Stokes equation \eqref{continue_generalized_stokes}.

\subsection{Finite element spaces and interpolation operators}
We will use the nonconforming linear element to discretize $\boldsymbol{\sigma} \in H_0^1(\Omega;\mathbb{S})$ and $ \boldsymbol{r} \in H^1(\Omega; \mathbb{R}^3)$, and piecewise constants to discretize $\boldsymbol{p} \in L^2(\Omega; \mathbb{T})$.

Recall the nonconforming linear element space \cite{CrouzeixRaviart1973}
\begin{equation*}
V_h^{\text{CR}} := \{ v_h \in \mathbb{P}_1(\mathcal{T}_h) : Q_{0,F} v_h \text{ is single-valued for all } F \in \Delta_2(\mathring{\mathcal{T}}_h)  \},
\end{equation*}
where $Q_{0,F}$ denotes the $L^2$-orthogonal projection operator onto $\mathbb{P}_0(F)$. The degree of freedom (DoF) is
\begin{equation}\label{CR_dof}
\int_{F} v \,\text{d}S, \quad \forall\,F \in \Delta_{2}(\mathcal{T}_h).
\end{equation}
Set $ V_h = V_h^{\text{CR}} \otimes \mathbb{R}^3$ and $\mathring{V}_h^{\mathbb{S}} = \mathring{V}_h^{\text{CR}} \otimes \mathbb{S} $, where
\begin{equation*}
\mathring{V}_h^{\text{CR}} = \{  v \in  V_h^{\text{CR}}: Q_{0,F} v = 0 \text{ for all } F \in \Delta_2(\mathcal{T}_h)\backslash \Delta_2(\mathring{\mathcal{T}}_h)  \}.
\end{equation*}
The spaces $ V_h  $ and $ \mathring{V}_h^{\mathbb{S}} $ satisfy the weak continuity conditions
\begin{align}
\label{eq:weakcontinuityCRvector}
\int_F [\boldsymbol{v}]\,\text{d}S &= 0, \quad \forall\,\boldsymbol{v} \in V_h, \, F \in \Delta_2(\mathring{\mathcal{T}}_h) , \\
\label{eq:weakcontinuityCRtensor}
\int_F [\boldsymbol{\tau}]\,\text{d}S &= 0, \quad \forall\,\boldsymbol{\tau} \in \mathring{V}_h^{\mathbb{S}}, \, F \in \Delta_2(\mathcal{T}_h).
\end{align}
The nonconforming linear element space $\mathring{V}_h^{\text{CR}}$ has the discrete Poincar\'e inequality
\begin{equation}\label{eq:PoincareCR}
\|v\|\lesssim |v|_{1,h}, \quad\forall\,v\in \mathring{V}_h^{\text{CR}}.
\end{equation}

Let $I_h: H^1(\Omega) \to  V_h^{\text{CR}}$ be the nodal interpolation operator based on the DoF~\eqref{CR_dof}. Its vector- and tensor-valued extensions are also denoted by $I_h$. Then we have for $T \in \mathcal{T}_h$ that 
\begin{align}
\int_T\nabla ( v - I_h v )\dx &= 0, \quad \forall\,v \in H^1(\Omega),  \label{interpolation_weak_continuty_cr}\\
\| v - I_h v \|_{0, T} + h_T | v - I_h v |_{1, T} &\lesssim h_T^s | v |_{s, T}, \quad \forall\,v \in H^s(\Omega), \, 1 \leq s \leq 2. \label{interpolation_error_cr}
\end{align}
Then combining trace inequality and \eqref{interpolation_error_cr} gives
\begin{equation}\label{interpolation_error_cr_jump}
\sum_{F \in \Delta_{2}(\mathcal{T}_h)} h_F^{-1}\| [  v - I_h v] \|_{F}^2 \lesssim
h^{2(s-1)} | v |_{s, h}^2, \quad \forall\,v \in H^s(\Omega), \, 1 \leq s \leq 2.
\end{equation}
Similar results hold for $I_h$ applied to vector- and tensor-valued functions.
Let $Q_h: L^2(\Omega; \mathbb{T}) \to \mathbb{P}_0(\mathcal{T}_h; \mathbb{T})$ be the $L^2$-orthogonal projector.
We have
\begin{equation*}
\| \boldsymbol{q} - Q_h \boldsymbol{q} \| \lesssim h | \boldsymbol{q} |_1,\quad \forall\,\boldsymbol{q} \in H^1(\Omega; \mathbb{T}).
\end{equation*}

\subsection{Broken Korn-type inequality}\label{sec:discrete_korn}

Using the technique in \cite{Brenner2004}, we can obtain the following broken Korn-type inequality for piecewise smooth vector-valued functions:
\begin{equation}\label{discrete_korn_rt}
\interleave \boldsymbol{v}\interleave_{1,h}^2 \eqsim
\| \dev\grad_h \boldsymbol{v} \|^2 + \sum_{F \in \Delta_{2}(\mathcal{T}_h)} h_F^{-1}\| [\boldsymbol{v}] \|_{F}^2,
\qquad \forall\,\boldsymbol{v}\in H^1(\mathcal{T}_h;\mathbb R^3).
\end{equation}

Next, we explain by means of two examples that $\| \dev\grad_h \boldsymbol{v} \|$ is not a norm on the space $V_h / \mathrm{RT}$ for general meshes. Equivalently, the nullity of the stiffness matrix associated with the bilinear form
\[
(\dev\grad_h \boldsymbol{r}, \dev\grad_h \boldsymbol{s}), \qquad \boldsymbol{r},\boldsymbol{s}\in V_h,
\]
can be larger than $\dim \mathrm{RT}=4$.

The mesh in Fig.~\ref{fig:mesh1} consists of two tetrahedra, and the corresponding stiffness matrix has nullity $5$. For the mesh in Fig.~\ref{fig:mesh3}, the nullity is $6$, and this number remains unchanged under uniform refinement. Both examples show that $\| \dev\grad_h \boldsymbol{v} \|$ is not necessarily a norm on $V_h / \mathrm{RT}$, thereby illustrating the necessity of the jump term in \eqref{discrete_korn_rt}.
\begin{figure}[htbp]
\centering
\vspace{-50pt}
\begin{subfigure}[b]{0.4\textwidth}
\centering
\adjustbox{trim=0.25\width 0.25\height 0.25\width 0.25\height, clip=true}
{\includegraphics[width=1.8\textwidth]{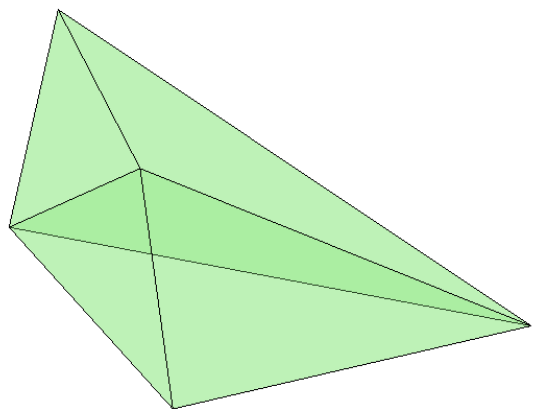}}
\caption{A two-tetrahedron mesh.}
\label{fig:mesh1}
\end{subfigure}
\vspace{0pt}
\begin{subfigure}[b]{0.4\textwidth}
\centering
\adjustbox{trim=0.25\width 0.25\height 0.25\width 0.25\height, clip=true}
{\includegraphics[width=1.8\textwidth]{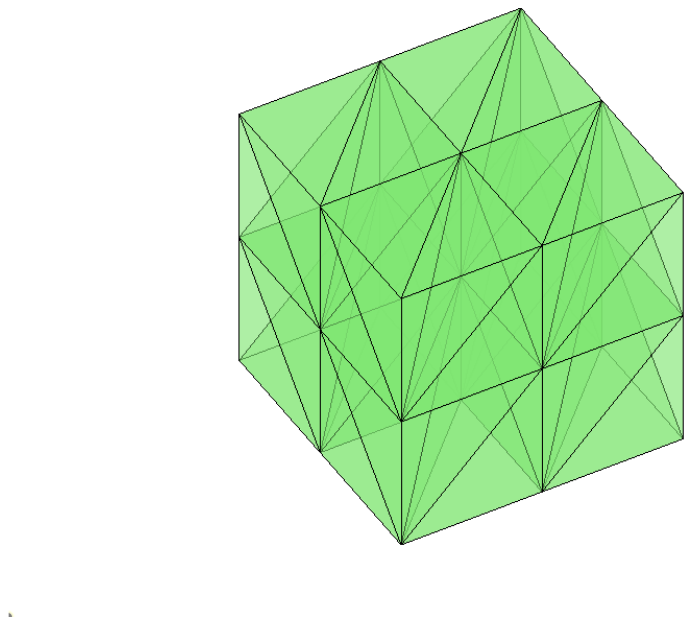}}
\caption{A uniform mesh on a cube.}
\label{fig:mesh3}
\end{subfigure}
\vspace{0pt}
\caption{Illustration of $\ker(\dev\grad_h)\cap V_h$ on different meshes. Its dimension is $5$ for the mesh in (A) and $6$ for the mesh in (B).}
\label{fig:meshes}
\end{figure}
\vspace{-10pt}

\subsection{Nonconforming linear element method}
Since the linear element space $V_h$ is $H^1$-nonconforming, the Korn-type inequality \eqref{Korninequality} does not hold for functions in $V_h$. Inspired by the broken Korn-type inequality \eqref{discrete_korn_rt}, we propose the following nonconforming linear element method for the generalized tensor-valued Stokes equation \eqref{continue_generalized_stokes}: Find
$(\boldsymbol{\sigma}_h, \boldsymbol{p}_h, \boldsymbol{r}_h) \in  \mathring{V}_h^{\mathbb{S}} \times \mathbb{P}_0(\mathcal{T}_h; \mathbb{T}) \times  V_h $ such that
\begin{subequations}\label{discrete_generalized_stokes}
\begin{align}
a_h(\boldsymbol{\sigma}_h, \boldsymbol{r}_h; \boldsymbol{\tau}, \boldsymbol{s}) + b_h(\boldsymbol{\tau}, \boldsymbol{s}; \boldsymbol{p}_h ) & = ( \boldsymbol{g}_h, \boldsymbol{\tau} ), && \forall\, \boldsymbol{\tau} \in \mathring{V}_h^{\mathbb{S}}, \boldsymbol{s} \in V_h , \label{discrete_generalized_stokes_1}        \\
b_h(\boldsymbol{\sigma}_h, \boldsymbol{r}_h; \boldsymbol{q} ) & = 0, && \forall\, \boldsymbol{q} \in \mathbb{P}_0(\mathcal{T}_h; \mathbb{T}), \label{discrete_generalized_stokes_2}
\end{align}
\end{subequations}
where $\boldsymbol{g}_h\in L^2(\Omega; \mathbb{S})$ is an approximation of $\boldsymbol{g}$, and the discrete bilinear forms
\begin{align*}
a_h(\boldsymbol{\sigma}, \boldsymbol{r}; \boldsymbol{\tau}, \boldsymbol{s}) & := \left( \nabla_h \boldsymbol{\sigma}, \nabla_h \boldsymbol{\tau} \right) + \sum_{F \in \Delta_{2}(\mathcal{T}_h) } h_F^{-1} \left( [\boldsymbol{r}], [\boldsymbol{s}] \right)_{F},  \\
b_h(\boldsymbol{\tau}, \boldsymbol{s}; \boldsymbol{p} ) & := \left( \curl_h \boldsymbol{\tau},  \boldsymbol{p} \right) + \left( \dev\grad _h \boldsymbol{s}, \boldsymbol{p} \right).
\end{align*}
The penalty term $\sum_{F \in \Delta_{2}(\mathcal{T}_h) } h_F^{-1} \left( [\boldsymbol{r}], [\boldsymbol{s}] \right)_{F}$ is incorporated into the bilinear form $a_h(\cdot, \cdot; \cdot, \cdot)$ to guarantee its discrete coercivity and the uniqueness of $\boldsymbol{r}_h$, as ensured by the broken Korn-type inequality \eqref{discrete_korn_rt}.

To facilitate the forthcoming analysis, we begin with a nonconforming discretization of the Helmholtz decomposition \eqref{continue_helm}.
\begin{lemma}
The following discrete Helmholtz decomposition holds:
\begin{equation}\label{discrete_helm0}
\mathbb{P}_0(\mathcal{T}_h; \mathbb{T}) = \textup{curl}_h \mathring{V}_h^{\mathbb{S}} + \dev\grad_h V_h.
\end{equation}
Concretely, for $\boldsymbol{q}_h \in \mathbb{P}_0(\mathcal{T}_h ; \mathbb{T})$, there exist $\boldsymbol{\tau}_h \in \mathring{V}_h^{\mathbb{S}} $ and $ \boldsymbol{s}_h \in V_h$ such that
\begin{equation}\label{discrete_helm}
\boldsymbol{q}_h = \textup{curl}_h \boldsymbol{\tau}_h + \dev\grad_h \boldsymbol{s}_h, \quad
\| \boldsymbol{\tau}_h \|_{1, h} + \interleave \boldsymbol{s}_h \interleave_{1, h} \lesssim \| \boldsymbol{q}_h \|.
\end{equation}
\end{lemma}
\begin{proof}
Applying the Helmholtz decomposition \eqref{continue_helm} to $\boldsymbol{q}_h$, there exist $ \boldsymbol{\tau} \in H_0^1(\Omega; \mathbb{S})$ and $\boldsymbol{s} \in H^1(\Omega; \mathbb{R}^3)/ \textrm{RT}$ such that
\begin{equation*}
\boldsymbol{q}_h = \curl \boldsymbol{\tau} + \dev\grad \boldsymbol{s}, \quad \| \boldsymbol{\tau} \|_1 + \| \boldsymbol{s} \|_1 \lesssim \| \boldsymbol{q}_h \|.
\end{equation*}
Let $ \boldsymbol{\tau}_h = I_h \boldsymbol{\tau}$, $\boldsymbol{s}_h = I_h \boldsymbol{s}$ and $\boldsymbol{p}_h$ any function in $\mathbb{P}_0(\mathcal{T}_h ; \mathbb{T}) $. Then by \eqref{interpolation_weak_continuty_cr}, we have on each $T \in \mathcal{T}_h$  that
\begin{equation}
\begin{aligned}
& \int_{T}(\curl\boldsymbol{\tau}_h + \dev\grad \boldsymbol{s}_h) : \boldsymbol{p}_h\dx  \\
= & \int_{T}(\curl\boldsymbol{\tau}   + \dev\grad \boldsymbol{s}):\boldsymbol{p}_h \dx = \int_{T}\boldsymbol{q}_h:\boldsymbol{p}_h \dx. \nonumber
\end{aligned}
\end{equation}
Consequently, $\boldsymbol{q}_h = \curl_h \boldsymbol{\tau}_h + \dev\grad_h \boldsymbol{s}_h$.
The proof is concluded using \eqref{interpolation_error_cr} and \eqref{interpolation_error_cr_jump}.
\end{proof}

\begin{remark}\rm
The decomposition \eqref{discrete_helm0} is not, in general, a direct sum. We refer to \cite[Corollary 6.4]{Schedensack2025} for some $L^2$-orthogonal decompositions of $\mathbb{P}_0(\mathcal{T}_h;\mathbb{T})$ and $\mathbb{P}_1(\mathcal{T}_h;\mathbb{T})$.
\end{remark}

With the discrete Helmholtz decomposition \eqref{discrete_helm0}, we show the following discrete coercivity and discrete inf-sup condition.
\begin{lemma}
The following inf-sup condition holds:
\begin{equation}\label{discrete_inf_sup}
\| \boldsymbol{q}_h \| \lesssim \sup\limits_{ \boldsymbol{\tau} \in \mathring{V}_h^{\mathbb{S}}, \boldsymbol{s} \in V_h  }\frac{ b_h \left( \boldsymbol{\tau} , \boldsymbol{s}; \boldsymbol{q}_h \right)}{  | \boldsymbol{\tau} |_{1, h} + \interleave \boldsymbol{s} \interleave_{1, h} }, \qquad\forall\,\boldsymbol{q}_h \in \mathbb{P}_0(\mathcal{T}_h; \mathbb{T}).
\end{equation}
\end{lemma}

\begin{proof}
It is a direct consequence of the discrete Helmholtz decomposition \eqref{discrete_helm}.
\end{proof}

\begin{lemma}
For $\left( \boldsymbol{\tau}, \boldsymbol{s}  \right) \in \textup{$ \mathring{V}_h^{\mathbb{S}}$} \times V_h $ satisfying $ b_h \left( \boldsymbol{\tau}, \boldsymbol{s}; \boldsymbol{q} \right) = 0 $ for all $\boldsymbol{q} \in \mathbb{P}_0(\mathcal{T}_h; \mathbb{T})$, we have
\begin{equation}\label{coer}
| \boldsymbol{\tau} |_{1, h}^2 + \interleave \boldsymbol{s} \interleave_{1, h}^2  \lesssim a_h\left(\boldsymbol{\tau}, \boldsymbol{s}; \boldsymbol{\tau}, \boldsymbol{s} \right).
\end{equation}
\end{lemma}
\begin{proof}
By the assumption, $\curl_h \boldsymbol{\tau} + \dev\grad _h \boldsymbol{s} = 0 $. It follows from the broken Korn-type inequality \eqref{discrete_korn_rt} that
\begin{align*}
| \boldsymbol{\tau} |_{1, h}^2 + \interleave \boldsymbol{s} \interleave_{1, h}^2
& \lesssim | \boldsymbol{\tau} |_{1, h}^2 + \| \dev\grad _h \boldsymbol{s} \|^2_{0} + \sum_{F \in \Delta_{2}(\mathcal{T}_h) } h_F^{-1}\| [\boldsymbol{s}] \|_{F}^2 \\
& = | \boldsymbol{\tau} |_{1, h}^2 + \| \curl_h \boldsymbol{\tau} \|^2_{0} +
\sum_{F \in \Delta_{2}(\mathcal{T}_h) } h_F^{-1}\| [\boldsymbol{s}] \|_{F}^2,
\end{align*}
which completes the proof.
\end{proof}

\begin{theorem}\label{thm:generalized_stokesfemwellposed}
The nonconforming linear element method \eqref{discrete_generalized_stokes} is well-posed, and $\curl_h \boldsymbol{\sigma}_h + \dev\grad _h \boldsymbol{r}_h=0$.
\end{theorem}
\begin{proof}
Employing the discrete inf-sup condition \eqref{discrete_inf_sup} and the discrete coercivity~\eqref{coer}, the well-posedness of the discrete method \eqref{discrete_generalized_stokes} follows from the Babu\v{s}ka-Brezzi theory~\cite{BoffiBrezziFortin2013}.
By the Helmholtz decomposition \eqref{discrete_helm}, $\curl_h \boldsymbol{\sigma}_h + \dev\grad _h \boldsymbol{r}_h=0$ follows from equation \eqref{discrete_generalized_stokes_2}.
\end{proof}

To the best of our knowledge, the only existing numerical method for the generalized tensor-valued Stokes equation \eqref{generalized_stokes} in three dimensions is the following nonconforming method proposed in \cite[Section 5.3.2]{Gallistl2017}: Find $\boldsymbol{\sigma}_h\in \mathring{V}_h^{\mathbb{S}}$ and $\boldsymbol{p}_h \in P_h$ such that
\begin{subequations}\label{gallistl_scheme}
\begin{align}
\left( \nabla_h \boldsymbol{\sigma}_h, \nabla_h \boldsymbol{\tau} \right) + \left( \curl_h \boldsymbol{\tau}, \boldsymbol{p}_h \right)
&= \left( \boldsymbol{g}_h, \boldsymbol{\tau} \right),
&& \forall\, \boldsymbol{\tau} \in \mathring{V}_h^{\mathbb{S}}, \\
\left( \curl_h \boldsymbol{\sigma}_h, \boldsymbol{q} \right)
&= 0,
&& \forall\, \boldsymbol{q} \in P_h,
\end{align}
\end{subequations}
where
$
P_h:=\mathbb{P}_0(\mathcal{T}_h;\mathbb{T})\cap H_0(\div,\Omega;\mathbb{T}).
$
However, since the space $P_h$ does not admit local degrees of freedom, its practical implementation remains unclear and, as noted in~\cite{ChenHuang2022,ChenHuang2024a}, may require vertex-associated degrees of freedom.

The space $P_h$ may be realized by introducing a Lagrange multiplier on each face to enforce normal continuity and a vanishing normal trace on the boundary. In this way, the method \eqref{gallistl_scheme} can be reformulated as follows: Find $( \boldsymbol{\sigma}_h, \boldsymbol{p}_h, \boldsymbol{\lambda}_h ) \in \mathring{V}_h^{\mathbb{S}} \times \mathbb{P}_0(\mathcal{T}_h; \mathbb{T}) \times \mathbb{P}_0(\mathcal{F}_h; \mathbb{R}^3) $ such that
\begin{subequations}\label{gallistl_equivalent_scheme}
\begin{align}
\left( \nabla_h \boldsymbol{\sigma}_h, \nabla_h \boldsymbol{\tau} \right) + \left(   \curl_h \boldsymbol{\tau},  \boldsymbol{p}_h  \right) + \sum_{F \in \mathcal{F}_h} (  [\boldsymbol{p}_h \boldsymbol{n}], \boldsymbol{\mu} )_F & = \left( \boldsymbol{g}_h, \boldsymbol{\tau} \right),   \\
\left(   \curl_h \boldsymbol{\sigma}_h,  \boldsymbol{q}  \right)  +  \sum_{F \in \mathcal{F}_h} (  [\boldsymbol{q}\boldsymbol{n}], \boldsymbol{\lambda}_h )_F & = 0,
\end{align}
for all $( \boldsymbol{\tau}, \boldsymbol{q}, \boldsymbol{\mu}) \in \mathring{V}_h^{\mathbb{S}} \times \mathbb{P}_0(\mathcal{T}_h; \mathbb{T}) \times \mathbb{P}_0(\mathcal{F}_h; \mathbb{R}^3)$,
\end{subequations}
where  $\mathcal{F}_h:=\Delta_2(\mathcal{T}_h)$ and
\[
\mathbb{P}_0(\mathcal{F}_h;\mathbb{R}^3)
:= \bigl\{ v \in L^2(\mathcal{F}_h;\mathbb{R}^3) : v|_F \in \mathbb{P}_0(F;\mathbb{R}^3)\; \text{for all } F \in \mathcal{F}_h \bigr\}.
\]

Let $Q_{\mathcal{F}_h}$ be the projection operator onto $\mathbb{P}_0(\mathcal{F}_h; \mathbb{R}^3)$. Since $Q_{\mathcal{F}_h}: V_h\to \mathbb{P}_0(\mathcal{F}_h; \mathbb{R}^3)$ is bijective, the method \eqref{gallistl_equivalent_scheme} is equivalent to the following scheme: Find
$(\boldsymbol{\sigma}_h, \boldsymbol{p}_h, \boldsymbol{r}_h) \in  \mathring{V}_h^{\mathbb{S}} \times \mathbb{P}_0(\mathcal{T}_h; \mathbb{T}) \times  V_h $ such that
\begin{subequations}\label{discrete_generalized_stokesGallistl}
\begin{align}
\left( \nabla_h \boldsymbol{\sigma}_h, \nabla_h \boldsymbol{\tau} \right) + b_h(\boldsymbol{\tau}, \boldsymbol{s}; \boldsymbol{p}_h ) & = ( \boldsymbol{g}_h, \boldsymbol{\tau} ), && \forall\, \boldsymbol{\tau} \in \mathring{V}_h^{\mathbb{S}}, \boldsymbol{s} \in V_h , \label{discrete_generalized_stokesGallistl_1}        \\
b_h(\boldsymbol{\sigma}_h, \boldsymbol{r}_h; \boldsymbol{q} ) & = 0, && \forall\, \boldsymbol{q} \in \mathbb{P}_0(\mathcal{T}_h; \mathbb{T}). \label{discrete_generalized_stokesGallistl_2}
\end{align}
\end{subequations}

As shown in Section~\ref{sec:discrete_korn}, $\| \dev\grad_h \boldsymbol{v} \|$ is not necessarily a norm on $V_h / \mathrm{RT}$ for general meshes. Consequently, the variable $\boldsymbol{r}_h$ in the scheme \eqref{discrete_generalized_stokesGallistl} is not necessarily unique, and the well-posedness of the scheme is therefore not guaranteed. By contrast, the jump term in our method \eqref{discrete_generalized_stokes} ensures discrete coercivity and the uniqueness of $\boldsymbol{r}_h$, thereby guaranteeing well-posedness.

\begin{remark}\rm
The use of a jump term to ensure the well-posedness of a nonconforming linear finite element method for elasticity was introduced in \cite{Hansbo2003}.
\end{remark}

\subsection{Error analysis}
We now proceed to the error analysis of the nonconforming linear element method~\eqref{discrete_generalized_stokes}.

\begin{lemma}\label{lem:20250902}
Let $\boldsymbol{\tau} = \boldsymbol{\sigma}_h - I_h \boldsymbol{\sigma}$.
Assume $\boldsymbol{\sigma} \in H^2(\Omega; \mathbb{S})$ and $\boldsymbol{p} \in H^1(\Omega; \mathbb{T})$.
We have
\begin{equation}\label{pre:estimate1}
( \sym\curl \boldsymbol{p} - \Delta \boldsymbol{\sigma}, \boldsymbol{\tau})
- ( \nabla_h(I_h \boldsymbol{\sigma}), \nabla_h \boldsymbol{\tau} ) \lesssim h | \boldsymbol{p} |_1 \interleave\boldsymbol{r}_h\interleave_{1, h} + h (| \boldsymbol{p} |_1 + | \boldsymbol{\sigma} |_2) |\boldsymbol{\tau}|_{1, h}.
\end{equation}
\end{lemma}
\begin{proof}
As $\curl \boldsymbol{\sigma} = 0$, using  \eqref{interpolation_weak_continuty_cr}  we have
\begin{equation*}
\curl_h (I_h \boldsymbol{\sigma}) = Q_{h}(\curl \boldsymbol{\sigma}) = 0.
\end{equation*}
It follows that
\begin{equation} \label{dis_equation1}
\curl_h \boldsymbol{\tau} + \dev\grad _h \boldsymbol{r}_h
= \curl_h \boldsymbol{\sigma}_h + \dev\grad _h \boldsymbol{r}_h = 0.
\end{equation}
By \eqref{dis_equation1}, $\boldsymbol{p} \in H_0(\div,\Omega; \mathbb{T}) \cap \ker(\div)$ (see Theorem \ref{gStokes:continue_wellposed}) and the weak continuity \eqref{eq:weakcontinuityCRvector} of $\boldsymbol{r}_h$, we obtain
\begin{align*}
\left( \boldsymbol{p}, \curl_h \boldsymbol{\tau} \right)
&= -\left( \boldsymbol{p}, \dev\grad _h \boldsymbol{r}_h \right)
= -\sum_{T \in \mathcal{T}_h} \left( \boldsymbol{p}\boldsymbol{n} , \boldsymbol{r}_h \right)_{\partial T} \\
&= \sum_{F \in \Delta_2 (\mathring{\mathcal{T}}_h)} \left(Q_{0,F} (\boldsymbol{p}\boldsymbol{n}) - \boldsymbol{p}\boldsymbol{n}, [\boldsymbol{r}_h] \right)_{F}
\lesssim h | \boldsymbol{p} |_1 \interleave\boldsymbol{r}_h\interleave_{1, h}.
\end{align*}
Similarly, by the weak continuity \eqref{eq:weakcontinuityCRtensor} of $\boldsymbol{\tau}$, we have
\begin{equation*}
-\sum_{T \in \mathcal{T}_h} \left( \boldsymbol{p} \times \boldsymbol{n} + \partial_n \boldsymbol{\sigma}, \boldsymbol{\tau}  \right)_{\partial T}
\lesssim h (| \boldsymbol{p} |_1 + | \boldsymbol{\sigma} |_2) |\boldsymbol{\tau}|_{1, h}. \end{equation*}
Combining the last two inequalities with integration by parts, we have
\begin{align*}
\left( \sym\curl \boldsymbol{p} - \Delta \boldsymbol{\sigma}, \boldsymbol{\tau}\right)
- \left( \nabla\boldsymbol{\sigma}, \nabla_h \boldsymbol{\tau} \right) & = \left( \boldsymbol{p}, \curl_h \boldsymbol{\tau} \right)
- \sum_{T \in \mathcal{T}_h} \left( \boldsymbol{p} \times \boldsymbol{n} + \partial_n \boldsymbol{\sigma}, \boldsymbol{\tau}  \right)_{\partial T} \\
&\lesssim h | \boldsymbol{p} |_1 \interleave\boldsymbol{r}_h\interleave_{1, h} + h (| \boldsymbol{p} |_1 + | \boldsymbol{\sigma} |_2) |\boldsymbol{\tau}|_{1, h}.
\end{align*}
Thus, we conclude \eqref{pre:estimate1} from the last inequality and the estimate \eqref{interpolation_error_cr} of $I_h$.
\end{proof}

\begin{theorem}\label{thm_H1bound}
Let $\left( \boldsymbol{\sigma}, \boldsymbol{p}, 0 \right) \in H_0^1(\Omega; \mathbb{S} ) \times L^2(\Omega; \mathbb{T}) \times H^1(\Omega; \mathbb{R}^3) $ be the solution of problem \eqref{continue_generalized_stokes}, and $ \left( \boldsymbol{\sigma}_h, \boldsymbol{p}_h, \boldsymbol{r}_h \right) \in \mathring{V}_h^{\mathbb{S}} \times \mathbb{P}_0(\mathcal{T}_h; \mathbb{T}) \times V_h $ be the solution of the discrete method \eqref{discrete_generalized_stokes}. Assume $\boldsymbol{\sigma} \in H^2(\Omega; \mathbb{S})$ and $\boldsymbol{p} \in H^1(\Omega; \mathbb{T})$. We have
\begin{equation}\label{H1bound}
|\boldsymbol{\sigma} - \boldsymbol{\sigma}_h |_{1,h} + \| \boldsymbol{p} - \boldsymbol{p}_h \| + \interleave\boldsymbol{r}_h\interleave_{1, h}
\lesssim h(| \boldsymbol{\sigma} |_2+| \boldsymbol{p} |_1) + \|\boldsymbol{g}-\boldsymbol{g}_h\|.
\end{equation}
\end{theorem}
\begin{proof}
Take $\boldsymbol{\tau} = \boldsymbol{\sigma}_h - I_h \boldsymbol{\sigma} $ and $\boldsymbol{s} = \boldsymbol{r}_h $ in \eqref{discrete_generalized_stokes_1}. By \eqref{dis_equation1},
$\dev\grad _h \boldsymbol{r}_h=-\curl_h \boldsymbol{\tau}$.
Consequently, equation \eqref{discrete_generalized_stokes_1} reduces to
\begin{equation}\label{H1sigma1}
\left( \nabla_h \boldsymbol{\sigma}_h , \nabla_h \boldsymbol{\tau} \right) + \sum_{F \in \Delta_{2}( \mathcal{T}_h) } h_F^{-1} \| [\boldsymbol{r}_h] \|_{F}^2 = \left(\boldsymbol{g}_h, \boldsymbol{\tau} \right),
\end{equation}
and we get from the broken Korn-type inequality \eqref{discrete_korn_rt} that
\begin{align*}
|\boldsymbol{\sigma}_h - I_h \boldsymbol{\sigma} |_{1, h}^2 + \interleave\boldsymbol{r}_h\interleave_{1, h}^2 \eqsim | \boldsymbol{\sigma}_h - I_h \boldsymbol{\sigma} |_{1, h}^2 + \sum_{F \in \Delta_{2}(\mathcal{T}_h) } h_F^{-1} \| [\boldsymbol{r}_h] \|_{F}^2.
\end{align*}
Then using \eqref{H1sigma1} and the fact $ -\Delta \boldsymbol{\sigma} + \sym\curl \boldsymbol{p} = \boldsymbol{g} $ from \eqref{generalized_stokes}, we have
\begin{align*}
|\boldsymbol{\sigma}_h - I_h \boldsymbol{\sigma} |_{1, h}^2 + & \interleave\boldsymbol{r}_h\interleave_{1, h}^2
\eqsim (\boldsymbol{g}_h , \boldsymbol{\tau})
- \left( \nabla_h(I_h \boldsymbol{\sigma}), \nabla_h \boldsymbol{\tau} \right) \\
& = \left( \sym\curl \boldsymbol{p} - \Delta \boldsymbol{\sigma}, \boldsymbol{\tau}\right)
- \left( \nabla_h(I_h \boldsymbol{\sigma}), \nabla_h \boldsymbol{\tau} \right) + (\boldsymbol{g}_h-\boldsymbol{g}, \boldsymbol{\tau}),
\end{align*}
which together with \eqref{pre:estimate1} implies
\begin{equation*}
| \boldsymbol{\sigma}_h - I_h \boldsymbol{\sigma} |_{1, h} + \interleave\boldsymbol{r}_h\interleave_{1, h}
\lesssim
h( |\boldsymbol{p} |_1 + | \boldsymbol{\sigma} |_2)  + \|\boldsymbol{g}-\boldsymbol{g}_h\|.
\end{equation*}
Combining this inequality with the interpolation estimate \eqref{interpolation_error_cr} yields
\begin{equation}\label{error_sigma_H1}
| \boldsymbol{\sigma} - \boldsymbol{\sigma}_h |_{1, h} + \interleave\boldsymbol{r}_h\interleave_{1, h}
\lesssim
h( |\boldsymbol{p} |_1 + | \boldsymbol{\sigma} |_2)  + \|\boldsymbol{g}-\boldsymbol{g}_h\|.
\end{equation}

On the other hand, for any $\boldsymbol{\tau} \in \mathring{V}_h^{\mathbb{S}} $ and $ \boldsymbol{s} \in V_h$, we get from \eqref{discrete_generalized_stokes_1} that
\begin{equation*}
\begin{aligned}
b_h\left( \boldsymbol{\tau}, \boldsymbol{s}; Q_h \boldsymbol{p} - \boldsymbol{p}_h \right)
&= (\curl_h \boldsymbol{\tau} + \dev\grad _h \boldsymbol{s}, \boldsymbol{p})
+ (\nabla_h \boldsymbol{\sigma}_h , \nabla_h \boldsymbol{\tau} ) - (\boldsymbol{g}_h, \boldsymbol{\tau}) \\
&\quad + \sum_{F \in \Delta_{2}( \mathcal{T}_h) } h_F^{-1} \left( [\boldsymbol{r}_h], [\boldsymbol{s}] \right)_{F} \\
&\lesssim (\curl_h \boldsymbol{\tau} + \dev\grad _h \boldsymbol{s}, \boldsymbol{p})
+ (\nabla_h \boldsymbol{\sigma}_h , \nabla_h \boldsymbol{\tau} ) - (\boldsymbol{g}, \boldsymbol{\tau}) \\
&\quad + \interleave\boldsymbol{r}_h\interleave_{1, h}\interleave\boldsymbol{s}\interleave_{1, h} + \|\boldsymbol{g}-\boldsymbol{g}_h\|\|\boldsymbol{\tau}\|.
\end{aligned}
\end{equation*}
Applying a similar argument to the proof of Lemma~\ref{lem:20250902},
and by the weak continuities~\eqref{eq:weakcontinuityCRvector}-\eqref{eq:weakcontinuityCRtensor}, we have
\begin{equation}\label{L2p2}
\begin{aligned}
& \quad\;
(\curl_h \boldsymbol{\tau} + \dev\grad _h \boldsymbol{s}, \boldsymbol{p})
+ \left( \nabla \boldsymbol{\sigma} , \nabla_h \boldsymbol{\tau} \right) - \left( \boldsymbol{g}, \boldsymbol{\tau} \right) \\
& =(\curl_h \boldsymbol{\tau} + \dev\grad _h \boldsymbol{s}, \boldsymbol{p})
+ \left( \nabla \boldsymbol{\sigma} , \nabla_h \boldsymbol{\tau} \right)
- \left(  \sym\curl \boldsymbol{p} , \boldsymbol{\tau} \right)
+ \left( \Delta \boldsymbol{\sigma}, \boldsymbol{\tau} \right) \\
&
= \sum_{T \in \mathcal{T}_h} \left( \partial_n \boldsymbol{\sigma} + \boldsymbol{p} \times \boldsymbol{n}  , \boldsymbol{\tau} \right)_{\partial T}
+ \sum_{T \in \mathcal{T}_h} \left(\boldsymbol{p}\boldsymbol{n}, \boldsymbol{s}\right)_{\partial T} \\
&\lesssim
h\left( | \boldsymbol{p} |_1 + | \boldsymbol{\sigma} |_2 \right) \left( | \boldsymbol{\tau} |_{1, h} + \interleave\boldsymbol{s}\interleave_{1,h} \right).
\end{aligned}
\end{equation}
It follows from \eqref{error_sigma_H1}-\eqref{L2p2} that
\begin{equation*}
b_h\left( \boldsymbol{\tau}, \boldsymbol{s}; Q_h \boldsymbol{p} - \boldsymbol{p}_h \right)
\lesssim
(h(|\boldsymbol{p} |_1 + | \boldsymbol{\sigma} |_2 ) + \|\boldsymbol{g}-\boldsymbol{g}_h\|)
\left( | \boldsymbol{\tau} |_{1, h} + \interleave\boldsymbol{s}\interleave_{1,h} \right).
\end{equation*}
Hence the discrete inf-sup condition \eqref{discrete_inf_sup} implies
\begin{equation*}
\| \boldsymbol{p} - \boldsymbol{p}_h \| \leq \| \boldsymbol{p} - Q_h \boldsymbol{p} \| + \| Q_h \boldsymbol{p} - \boldsymbol{p}_h \|
\lesssim
h \left( |\boldsymbol{p} |_1 + | \boldsymbol{\sigma} |_2 \right) + \|\boldsymbol{g}-\boldsymbol{g}_h\|.
\end{equation*}
This together with \eqref{error_sigma_H1} gives \eqref{H1bound}.
\end{proof}


We then use the duality argument to estimate $\| \boldsymbol{\sigma} - \boldsymbol{\sigma}_h \| $. Let $ \hat{u} \in H_0^3(\Omega) $ be the solution of the following dual problem
\begin{equation} \label{dual_problem}
- \Delta^3 \hat{u} = \div\div \left( \boldsymbol{\sigma} - \boldsymbol{\sigma}_h \right).
\end{equation}
Assume that the dual problem \eqref{dual_problem} admits the regularity estimate
\begin{equation}\label{dual_regu}
\| \hat{u} \|_4 \lesssim \| \boldsymbol{\sigma} - \boldsymbol{\sigma}_h \|.
\end{equation}
For regularity results of polyharmonic equations we refer to \cite{MazyaRossmann2010,KozlovMazyaRossmann2001}.
In addition, we assume that the space $H_0(\div,\Omega; \mathbb{T})\cap H(\sym\curl,\Omega; \mathbb{T})$ is continuously embedded into $H^1(\Omega; \mathbb{T})$, namely,
\begin{equation}\label{traceless_regu}
\|\boldsymbol{q}\|_1 \lesssim \|\div\boldsymbol{q}\|+\|\sym\curl\boldsymbol{q}\|, \quad \forall\,\boldsymbol{q}\in H_0(\div,\Omega; \mathbb{T})\cap H(\sym\curl,\Omega; \mathbb{T}).
\end{equation}
Comparable regularity properties for vector functions on convex domains can be found in \cite[Section~3.5]{GiraultRaviart1986} and \cite[Corollary~5.2]{Mitrea2001}.

\begin{lemma}
Assume that the regularity conditions \eqref{dual_regu}-\eqref{traceless_regu} hold. Then there exist  $ \hat{\boldsymbol{\sigma}} \in H_0^1(\Omega;\mathbb{S}) \cap H^2(\Omega; \mathbb{S}) $ and $ \hat{\boldsymbol{p}} \in H_0(\div, \Omega; \mathbb{T})\cap H^1(\Omega; \mathbb{T}) $ such that
\begin{align}
-\Delta \hat{\boldsymbol{\sigma}} + \sym\curl\hat{\boldsymbol{p}}  = \boldsymbol{\sigma} - \boldsymbol{\sigma}_h,
\;\;  \curl\hat{\boldsymbol{\sigma}} = 0, \;\; \div\hat{\boldsymbol{p}} = 0,  \label{dual_replace} \\
\| \hat{\boldsymbol{\sigma}} \|_2 + \| \hat{\boldsymbol{p}} \|_1  \lesssim \|\boldsymbol{\sigma} - \boldsymbol{\sigma}_h\|. \qquad\qquad\qquad \label{dual_estimate}
\end{align}
\end{lemma}
\begin{proof}
Let $\hat{\boldsymbol{\sigma}} = \nabla^2 \hat{u} \in H_0^1(\Omega; \mathbb{S}) \cap \text{ker(curl)} $, then \eqref{dual_problem} can be written as
\begin{equation*}
- \div\div \Delta \hat{\boldsymbol{\sigma}} = \div\div ( \boldsymbol{\sigma} - \boldsymbol{\sigma}_h ),
\end{equation*}
or equivalently
\begin{equation*}
\div\div\left( \Delta \hat{\boldsymbol{\sigma}} + \boldsymbol{\sigma} - \boldsymbol{\sigma}_h \right) = 0.
\end{equation*}
By the divdiv complex \cite{ChenHuang2022a,ArnoldHu2021}
\begin{equation*}
\fontsize{5pt}{6pt}\selectfont
H^1(\Omega;\mathbb R^3) \xrightarrow{\dev\grad} H(\sym\curl,\Omega;\mathbb T) \xrightarrow{\sym\curl}  H(\div\div,\Omega;\mathbb S)\xrightarrow{\div\div}L^2(\Omega) \xrightarrow{}0,
\end{equation*}
there exists a $\hat{\boldsymbol{p}} \in H(\sym\curl,\Omega;\mathbb T)\cap H_0(\div, \Omega; \mathbb{T})$ satisfying
\begin{equation*}
\sym\curl \hat{\boldsymbol{p}} = \Delta \hat{\boldsymbol{\sigma}} + \boldsymbol{\sigma} - \boldsymbol{\sigma}_h, \quad \div\hat{\boldsymbol{p}} = 0,
\end{equation*}
which combined with the assumption \eqref{traceless_regu} indicates
\begin{equation*}
\| \hat{\boldsymbol{p}} \|_1 \lesssim \| \Delta \hat{\boldsymbol{\sigma}} + \boldsymbol{\sigma} - \boldsymbol{\sigma}_h \| \lesssim \| \hat{\boldsymbol{\sigma}} \|_2 + \| \boldsymbol{\sigma} - \boldsymbol{\sigma}_h \|.
\end{equation*}
Then \eqref{dual_replace} is true, and \eqref{dual_estimate} follows from \eqref{dual_regu}.
\end{proof}

\begin{lemma}
Assume that the regularity conditions \eqref{dual_regu}-\eqref{traceless_regu} are satisfied. Under the assumptions of Theorem~\ref{thm_H1bound}, we have
\begin{equation}\label{pre:dualestimate1}
(\boldsymbol{\sigma} - \boldsymbol{\sigma}_h , \sym\curl \hat{\boldsymbol{p}}) \lesssim   \|\boldsymbol{\sigma} - \boldsymbol{\sigma}_h\|\big(h^2(| \boldsymbol{\sigma} |_2 + |\boldsymbol{p}|_1) + h\|\boldsymbol{g}-\boldsymbol{g}_h\|\big).
\end{equation}
\end{lemma}
\begin{proof}
Using integration by parts, we get from \eqref{dis_equation1} and \eqref{dual_replace} that
\begin{align*}
(\boldsymbol{\sigma} - \boldsymbol{\sigma}_h , \sym\curl \hat{\boldsymbol{p}})
& = (\curl_h \left(\boldsymbol{\sigma} - \boldsymbol{\sigma}_h\right), \hat{\boldsymbol{p}}) - \sum_{T \in \mathcal{T}_h} (\hat{\boldsymbol{p}} \times \boldsymbol{n}, \boldsymbol{\sigma} - \boldsymbol{\sigma}_h)_{\partial T} \\
&= (\dev\grad _h \boldsymbol{r}_h, \hat{\boldsymbol{p}} )
- \sum_{T \in \mathcal{T}_h} (\hat{\boldsymbol{p}} \times \boldsymbol{n}, \boldsymbol{\sigma} - \boldsymbol{\sigma}_h)_{\partial T} \\
&= \sum_{T \in \mathcal{T}_h} (\hat{\boldsymbol{p}}\boldsymbol{n}, \boldsymbol{r}_h)_{\partial T}
- \sum_{T \in \mathcal{T}_h} (\hat{\boldsymbol{p}} \times \boldsymbol{n}, \boldsymbol{\sigma} - \boldsymbol{\sigma}_h)_{\partial T}.
\end{align*}
Applying a similar argument to the proof of Lemma~\ref{lem:20250902},
and by the weak continuities~\eqref{eq:weakcontinuityCRvector}-\eqref{eq:weakcontinuityCRtensor},
we have
\begin{equation*}
(\boldsymbol{\sigma} - \boldsymbol{\sigma}_h , \sym\curl \hat{\boldsymbol{p}}) \lesssim h|\hat{\boldsymbol{p}}|_1(|\boldsymbol{\sigma} - \boldsymbol{\sigma}_h |_{1,h} + \interleave\boldsymbol{r}_h\interleave_{1, h}).
\end{equation*}
Thus, the estimate \eqref{pre:dualestimate1} follows from \eqref{H1bound} and \eqref{dual_estimate}.
\end{proof}

\begin{lemma}
Assume that the regularity conditions \eqref{dual_regu}-\eqref{traceless_regu} are satisfied. Under the assumptions of Theorem~\ref{thm_H1bound}, we have
\begin{equation}\label{pre:dualestimate2}
( \boldsymbol{\sigma} - \boldsymbol{\sigma}_h , -\Delta \hat{\boldsymbol{\sigma}}) \lesssim   \|\boldsymbol{\sigma} - \boldsymbol{\sigma}_h\|\big(h^2(| \boldsymbol{\sigma} |_2+|\boldsymbol{p} |_1) + h\|\boldsymbol{g}-\boldsymbol{g}_h\|+\|\boldsymbol{g}-\boldsymbol{g}_h\|_{-1}\big).
\end{equation}
\end{lemma}
\begin{proof}
Using $ -\Delta \boldsymbol{\sigma} + \sym\curl \boldsymbol{p} = \boldsymbol{g} $ and \eqref{discrete_generalized_stokes_1} with $ \boldsymbol{s} = 0$ and $\boldsymbol{\tau}=I_h\hat{\boldsymbol{\sigma}}$, we have

\begin{equation}\label{dualerror3}
\begin{aligned}
& \quad\;
(\nabla_h ( \boldsymbol{\sigma} - \boldsymbol{\sigma}_h), \nabla_h(I_h\hat{\boldsymbol{\sigma}})) \\
&= ( \nabla \boldsymbol{\sigma}, \nabla_h(I_h\hat{\boldsymbol{\sigma}}))
- ( \nabla_h \boldsymbol{\sigma}_h, \nabla_h(I_h\hat{\boldsymbol{\sigma}}))   \\
&= ( \nabla \boldsymbol{\sigma}, \nabla_h(I_h\hat{\boldsymbol{\sigma}})) + ( \boldsymbol{p}_h, \curl_h(I_h\hat{\boldsymbol{\sigma}})) - (\boldsymbol{g}_h, I_h \hat{\boldsymbol{\sigma}}) \\
&= (\nabla \boldsymbol{\sigma}, \nabla_h(I_h\hat{\boldsymbol{\sigma}}))+ ( \boldsymbol{p}_h, \curl_h(I_h\hat{\boldsymbol{\sigma}}))
+ (\Delta\boldsymbol{\sigma} - \sym\curl\boldsymbol{p} + \boldsymbol{g}-\boldsymbol{g}_h, I_h\hat{\boldsymbol{\sigma}})  \\
&= (\boldsymbol{p}_h - \boldsymbol{p} , \curl_h(I_h \hat{\boldsymbol{\sigma}})) + (\boldsymbol{g}-\boldsymbol{g}_h, I_h \hat{\boldsymbol{\sigma}})
+ \sum_{T \in \mathcal{T}_h}( \partial_n \boldsymbol{\sigma} + \boldsymbol{p} \times \boldsymbol{n} , I_h \hat{\boldsymbol{\sigma}})_{\partial T}.
\end{aligned}
\end{equation}

Using the interpolation estimate \eqref{interpolation_error_cr},
\begin{align*}
& \quad\;
(\boldsymbol{p}_h - \boldsymbol{p}, \curl_h(I_h \hat{\boldsymbol{\sigma}})) + (\boldsymbol{g}-\boldsymbol{g}_h, I_h \hat{\boldsymbol{\sigma}}) \\
& = (\boldsymbol{p}_h - \boldsymbol{p}, \curl_h(I_h \hat{\boldsymbol{\sigma}}-\hat{\boldsymbol{\sigma}})) + (\boldsymbol{g}-\boldsymbol{g}_h, I_h \hat{\boldsymbol{\sigma}}) \\
& \lesssim  h\|\hat{\boldsymbol{\sigma}}\|_2(\|\boldsymbol{p} - \boldsymbol{p}_h\| + h\|\boldsymbol{g}-\boldsymbol{g}_h\|) +\|\boldsymbol{g}-\boldsymbol{g}_h\|_{-1}|\hat{\boldsymbol{\sigma}}|_1.
\end{align*}
By the weak continuity \eqref{eq:weakcontinuityCRtensor} and the interpolation estimate \eqref{interpolation_error_cr},
\begin{equation*}
\sum_{T \in \mathcal{T}_h}( \partial_n \boldsymbol{\sigma} + \boldsymbol{p} \times \boldsymbol{n} , I_h \hat{\boldsymbol{\sigma}})_{\partial T}\lesssim h^2(|\boldsymbol{\sigma}|_2+|\boldsymbol{p}|_1)|\hat{\boldsymbol{\sigma}}|_2.
\end{equation*}
Inserting the last two inequalities into \eqref{dualerror3}, we obtain
\begin{align*}
(\nabla_h ( \boldsymbol{\sigma} - \boldsymbol{\sigma}_h), \nabla_h(I_h\hat{\boldsymbol{\sigma}}))&\lesssim h\|\hat{\boldsymbol{\sigma}}\|_2(\|\boldsymbol{p} - \boldsymbol{p}_h\| + \|\boldsymbol{g}-\boldsymbol{g}_h\| + h(|\boldsymbol{\sigma}|_2+|\boldsymbol{p}|_1)) \\
&\quad+\|\boldsymbol{g}-\boldsymbol{g}_h\|_{-1}|\hat{\boldsymbol{\sigma}}|_1.
\end{align*}
Then we get from the interpolation estimate \eqref{interpolation_error_cr} and \eqref{H1bound} that
\begin{align*}
(\nabla_h ( \boldsymbol{\sigma} - \boldsymbol{\sigma}_h), \nabla_h\hat{\boldsymbol{\sigma}})&=(\nabla_h ( \boldsymbol{\sigma} - \boldsymbol{\sigma}_h), \nabla_h(\hat{\boldsymbol{\sigma}}-I_h\hat{\boldsymbol{\sigma}}))+(\nabla_h ( \boldsymbol{\sigma} - \boldsymbol{\sigma}_h), \nabla_h(I_h\hat{\boldsymbol{\sigma}})) \\
&\lesssim \|\hat{\boldsymbol{\sigma}}\|_2\big(h^2(| \boldsymbol{\sigma} |_2+|\boldsymbol{p} |_1) + h\|\boldsymbol{g}-\boldsymbol{g}_h\|+\|\boldsymbol{g}-\boldsymbol{g}_h\|_{-1}\big).
\end{align*}
This together with \eqref{dual_estimate} gives
\begin{equation*}
(\nabla_h ( \boldsymbol{\sigma} - \boldsymbol{\sigma}_h), \nabla_h\hat{\boldsymbol{\sigma}}) \lesssim   \|\boldsymbol{\sigma} - \boldsymbol{\sigma}_h\|\big(h^2(| \boldsymbol{\sigma} |_2+|\boldsymbol{p} |_1) + h\|\boldsymbol{g}-\boldsymbol{g}_h\|+\|\boldsymbol{g}-\boldsymbol{g}_h\|_{-1}\big).
\end{equation*}
On the other hand, by the weak continuity \eqref{eq:weakcontinuityCRtensor}, \eqref{dual_estimate} and \eqref{H1bound},
\begin{align*}
- \sum_{T \in \mathcal{T}_h}(\partial_n \hat{\boldsymbol{\sigma}}, \boldsymbol{\sigma} - \boldsymbol{\sigma}_h )_{\partial T}
& \lesssim h|\hat{\boldsymbol{\sigma}}|_2|\boldsymbol{\sigma} - \boldsymbol{\sigma}_h|_{1,h} \\
&\lesssim \|\boldsymbol{\sigma} - \boldsymbol{\sigma}_h\|\big(h^2(| \boldsymbol{\sigma} |_2+|\boldsymbol{p} |_1) + h\|\boldsymbol{g}-\boldsymbol{g}_h\|\big).
\end{align*}
Finally, combining the last two inequalities yields \eqref{pre:dualestimate2}.
\end{proof}

\begin{theorem}
Assume that the regularity conditions \eqref{dual_regu}-\eqref{traceless_regu} are satisfied. Under the assumptions of Theorem~\ref{thm_H1bound}, we have
\begin{equation} \label{dual_error}
\| \boldsymbol{\sigma} - \boldsymbol{\sigma}_h \| \lesssim
h^2(| \boldsymbol{\sigma} |_2+|\boldsymbol{p} |_1) + h\|\boldsymbol{g}-\boldsymbol{g}_h\| + \|\boldsymbol{g}-\boldsymbol{g}_h\|_{-1}.
\end{equation}
\end{theorem}

\begin{proof}
Using \eqref{dual_replace}, \eqref{pre:dualestimate1} and \eqref{pre:dualestimate2}, we have
\begin{align*}
\| \boldsymbol{\sigma} - \boldsymbol{\sigma}_h \|^2 &= ( \boldsymbol{\sigma} - \boldsymbol{\sigma}_h , -\Delta \hat{\boldsymbol{\sigma}} + \sym\curl \hat{\boldsymbol{p}}) \\
&\lesssim  \|\boldsymbol{\sigma} - \boldsymbol{\sigma}_h\|\big(h^2(| \boldsymbol{\sigma} |_2+|\boldsymbol{p} |_1) + h\|\boldsymbol{g}-\boldsymbol{g}_h\|+\|\boldsymbol{g}-\boldsymbol{g}_h\|_{-1}\big).
\end{align*}
Therefore, \eqref{dual_error} holds.
\end{proof}

\section{Decoupled Finite Element Method for Triharmonic Equation}\label{section4}

In this section, we develop and analyze a low-order decoupled finite element method for the three-dimensional triharmonic equation. We first decouple the triharmonic equation into two biharmonic equations and one generalized tensor-valued Stokes equation \eqref{continue_generalized_stokes}. This decoupled formulation facilitates the construction of efficient finite element methods and the design of fast solvers.
The two biharmonic equations are then discretized using the Morley--Wang--Xu element \cite{morley2006}, while the generalized tensor-valued Stokes equation is approximated by the nonconforming linear element method \eqref{discrete_generalized_stokes}.

\subsection{Decoupled formulation}
The primal formulation of the triharmonic equation~\eqref{triharmonic} reads: Find $u \in H_0^3(\Omega) $ such that
\begin{equation}\label{triharmonic_primal}
(\nabla^{3} u, \nabla^{3} v)=(f, v) ,\quad \forall\,v \in H_{0}^{3}(\Omega).
\end{equation}

To derive a decoupled formulation for the primal formulation \eqref{triharmonic_primal}, we recall the following divdiv complex in three dimensions \cite[(35)]{ArnoldHu2021}:
\begin{equation*}
\textrm{RT} \xrightarrow{\subset} H^1(\Omega; \mathbb{R}^3) \xrightarrow{\dev\grad} L^2(\Omega; \mathbb{T}) \xrightarrow{\sym\curl} H^{-1}(\Omega; \mathbb{S}) \xrightarrow{\div\div} H^{-3}(\Omega) \xrightarrow{ \, \, } 0.
\end{equation*}
Applying the tilde operation in \cite[Section 2.3]{ChenHuang2025} to the last divdiv complex yields the following exact divdiv complex
\begin{equation}\label{complex_div_negative}
\begin{aligned}
\textrm{RT} \xrightarrow{\subset} H^1(\Omega; \mathbb{R}^3) \xrightarrow{\dev\grad} & L^2(\Omega; \mathbb{T}) \xrightarrow{\sym\curl} \\
&H^{-2}(\div\div, \Omega; \mathbb{S}) \xrightarrow{\div\div}  H^{-2}(\Omega) \xrightarrow{ \, \, } 0,
\end{aligned}
\end{equation}
where
\begin{equation*}
H^{-2}(\div\div, \Omega; \mathbb{S}) := \{\boldsymbol{\tau} \in H^{-1}(\Omega; \mathbb{S})
: \div\div\boldsymbol{\tau} \in H^{-2}(\Omega) \}.
\end{equation*}

Using the divdiv complex \eqref{complex_div_negative} above, we construct the following commutative diagram
\begin{equation}\label{cd}
\begin{array}{c}
\xymatrix{
H_{0}^{1}( \Omega; \mathbb{S} ) \ar[r]^-{\Delta} &  H^{-1}(\Omega;\mathbb{S} )
\\
L^{2}(\Omega;\mathbb{T} ) \ar[r]^- { \sym\curl } &  H^{-2}(\div\div ,\Omega;\mathbb S) \ar@{}[u]|{\bigcup}
\ar[r] ^-{  \div\div } & H^{-2}(\Omega)  \ar[r]^-{} & 0. \\
& \ar[u]^{\boldsymbol{I}} H_{0}(\textrm{curl}, \Omega; \mathbb{S}) & \ar[l]_-{\nabla^2 } H_{0}^{2}(\Omega)
\ar[u]_{\Delta^2}
}
\end{array}
\end{equation}
Then, by applying the framework in \cite{ChenHuang2018} to the commutative diagram \eqref{cd}, we obtain the Helmholtz decomposition
\begin{equation*}
H^{-2}(\div\div, \Omega;\mathbb S) =  \nabla^2 H_0^2(\Omega)
\oplus \sym\curl ( L^{2}(\Omega;\mathbb{T} )/ \dev\grad H^{1}(\Omega; \mathbb{R}^3) ),
\end{equation*}
and decouple the triharmonic equation \eqref{triharmonic_primal} into the following three equations: Find $(w, \boldsymbol{\sigma}, \boldsymbol{p}, \boldsymbol{r} , u )  \in H_0^2(\Omega) \times H_0^1(\Omega; \mathbb{S}) \times L^2(\Omega; \mathbb{T}) \times (H^1(\Omega; \mathbb{R}^3) / \textrm{RT}) \times H_0^2(\Omega) $ such that
\begin{subequations}\label{continue}
\begin{align}
( \nabla^2 w, \nabla^2 v ) & = (f, v ),  &&  \label{continue_1} \\
( \nabla \boldsymbol{\sigma}, \nabla \boldsymbol{\tau} ) + ( \curl \boldsymbol{\tau} + \dev\grad \boldsymbol{s}  , \boldsymbol{p} ) & = ( \nabla^2 w, \boldsymbol{\tau} ), &&  \label{continue_2} \\
\left(\curl \boldsymbol{\sigma} + \dev\grad \boldsymbol{r}, \boldsymbol{q} \right) & = 0, && \label{continue_3} \\
( \nabla^2 u, \nabla^2 \chi) & = ( \boldsymbol{\sigma} , \nabla^2 \chi) , && \label{continue_4}
\end{align}
\end{subequations}
for any $(v, \boldsymbol{\tau}, \boldsymbol{q}, \boldsymbol{s}, \chi)  \in H_0^2(\Omega) \times H_0^1(\Omega; \mathbb{S}) \times L^2(\Omega; \mathbb{T}) \times (H^1(\Omega; \mathbb{R}^3) / \textrm{RT}) \times H_0^2(\Omega)$.

Problems \eqref{continue_1} and \eqref{continue_4} are the weak formulations of the biharmonic equations, which are evidently well-posed.
By Theorem~\ref{gStokes:continue_wellposed}, the generalized tensor-valued Stokes equation \eqref{continue_2}-\eqref{continue_3} is also well-posed.
Next, we show the equivalence between the decoupled formulation \eqref{continue} and the primal formulation \eqref{triharmonic_primal}.

\begin{theorem}\label{theorem_equivalence}
The decoupled formulation \eqref{continue} is equivalent to the primal formulation \eqref{triharmonic_primal}. That is, if $w \in H_0^2(\Omega)$ is the solution of problem \eqref{continue_1}, $(\boldsymbol{\sigma}, \boldsymbol{p}, \boldsymbol{r}) \in H_0^1(\Omega; \mathbb{S}) \times L^2(\Omega; \mathbb{T}) \times (H^1(\Omega; \mathbb{R}^3)/\textnormal{RT}) $ is the solution of problem \eqref{continue_2}-\eqref{continue_3}, and $u \in H_0^2(\Omega)$ is the solution of problem \eqref{continue_4}, then $\boldsymbol{r} = 0$, $\div\boldsymbol{p} = 0$, $\boldsymbol{\sigma} = \nabla^2 u $, $\boldsymbol{p} \in H_0(\div, \Omega;\mathbb T)$, and $ u \in H_0^3(\Omega) $ satisfies the primal formulation \eqref{triharmonic_primal}.
\end{theorem}
\begin{proof}
%
Combining equation \eqref{continue_3} with the Helmholtz decomposition \eqref{continue_helm} yields $ \boldsymbol{r} = 0 $ and $\curl\boldsymbol{\sigma} = 0$. By the Hessian complex \cite{ChenHuang2022,ArnoldHu2021,PaulyZulehner2020}, there exists $\tilde{u} \in H_0^3(\Omega)$ such that $ \boldsymbol{\sigma}  = \nabla^2 \tilde{u} \in H_0^1(\Omega; \mathbb{S})$.  Then, by \eqref{continue_4}, we have $u = \tilde{u} \in H_0^3(\Omega)$ and $\boldsymbol{\sigma} = \nabla^2 u$.
Next, taking $\boldsymbol{s} = 0$ and $\boldsymbol{\tau} = \nabla^2 v$ with $v \in H_0^3(\Omega)$ in \eqref{continue_2}, we obtain
\begin{equation*}
( \nabla^3 u, \nabla^3 v) = ( \nabla^2 w, \nabla^2 v ), \qquad\forall\,v \in H_0^3(\Omega) .
\end{equation*}
Together with \eqref{continue_1}, this shows that $u \in H_0^3(\Omega)$ satisfies the primal formulation~\eqref{triharmonic_primal}. Finally, choosing $\boldsymbol{\tau} = 0$ in \eqref{continue_2} gives $\boldsymbol{p} \in H_0(\div, \Omega;\mathbb T)$ and $\div \boldsymbol{p} = 0$.
\end{proof}

The decoupled formulation \eqref{continue} is different from the one presented in \cite[Section 5.1]{Gallistl2017}.

\subsection{Morley-Wang-Xu element}
For the three-dimensional biharmonic equations \eqref{continue_1} and \eqref{continue_4}, various numerical approaches are available, including conforming element methods \cite{HuLinWu2024,ChenChenGaoHuangEtAl2025,ChenHuang2021,ChenHuang2024,ChenHuangWei2022}, nonconforming element methods \cite{morley2006,ChenHuang2020}, mixed finite element methods \cite{ChenHuang2025a,HuangTang2025}, and decoupling methods \cite{CuiHuang2025,CuiHuang2024,Cui2025,Gallistl2017,Zhang2018a}.
In this paper, we employ the nonconforming Morley-Wang-Xu element method \cite{morley2006}, whose equivalence to the $H(\div\div)$-conforming mixed element method was established in our recent work \cite{ChenHuang2025a}. The shape function space for the Morley-Wang-Xu element is $\mathbb{P}_2(T)$, and the DoFs are given by
\begin{subequations}\label{dof_w}
\begin{align}
\int_{F} \partial_n w \, \text{d} S, \quad & \forall\,F \in \Delta_2(T) , \label{dof_w_1} \\
\int_{e} w \, \text{d} s, \quad & \forall\,e \in \Delta_1(T).  \label{dof_w_2}
\end{align}
\end{subequations}
The global $H^2$-nonconforming finite element space is defined as
\begin{equation*}
\mathring{W}_h = \{ w_h \in W_h: \text{ DoFs \eqref{dof_w} vanish on the boundary } \partial\Omega \},
\end{equation*}
where
\begin{equation*}
W_h = \{ w_h\in\mathbb{P}_2(\mathcal{T}_h): \text{ DoFs \eqref{dof_w} are single-valued}    \}.
\end{equation*}
The Morley-Wang-Xu space $\mathring{W}_h$ satisfies the following weak continuity property:
\begin{equation}
\label{eq:weakcontinuityMWX}
\int_F [\nabla_hw_h]\,\text{d}S = 0, \quad \forall\,w_h \in \mathring{W}_h, \, F \in \Delta_2(\mathcal{T}_h).
\end{equation}
This implies $\nabla_h \mathring{W}_h\subseteq \mathring{V}_h^{\text{CR}}\otimes \mathbb{R}^3$.


\subsection{Decoupled finite element method}
With all ingredients prepared, we now proceed to construct the following decoupled nonconforming finite element method for the three-dimensional triharmonic equation \eqref{triharmonic_primal} based on the decoupled formulation \eqref{continue}: Find
$(w_h, \boldsymbol{\sigma}_h, \boldsymbol{p}_h, \boldsymbol{r}_h, u_h) \in \mathring{W}_h \times \mathring{V}_h^{\mathbb{S}} \times \mathbb{P}_0(\mathcal{T}_h; \mathbb{T}) \times  V_h \times \mathring{W}_h$ such that
\begin{subequations}\label{decoupledFEM}
\begin{align}
\left( \nabla^2_h w_h, \nabla^2_h v \right) & = \left(f, v \right)  && \forall\,v \in \mathring{W}_h, \label{discrete_1} \\
a_h(\boldsymbol{\sigma}_h, \boldsymbol{r}_h; \boldsymbol{\tau}, \boldsymbol{s}) + b_h(\boldsymbol{\tau}, \boldsymbol{s}; \boldsymbol{p}_h ) & = \left( \nabla_h^2 w_h, \boldsymbol{\tau} \right) && \forall\,\boldsymbol{\tau} \in \mathring{V}_h^{\mathbb{S}}, \boldsymbol{s} \in V_h, \label{discrete_2}        \\
b_h(\boldsymbol{\sigma}_h, \boldsymbol{r}_h; \boldsymbol{q} ) & = 0 && \forall\,\boldsymbol{q} \in \mathbb{P}_0(\mathcal{T}_h; \mathbb{T}), \label{discrete_3} \\
\left( \nabla_h^2 u_h, \nabla_h^2 \chi \right) & = \left( \boldsymbol{\sigma}_h , \nabla_h^2 \chi \right) && \forall\,\chi \in  \mathring{W}_h. \label{discrete_4}
\end{align}
\end{subequations}
That is, we use the Morley-Wang-Xu element method to discretize the biharmonic equations \eqref{continue_1} and \eqref{continue_4}, and the nonconforming linear element method \eqref{discrete_generalized_stokes} to discretize the generalized tensor-valued Stokes equation \eqref{continue_2}-\eqref{continue_3}.

The well-posedness of the Morley-Wang-Xu element methods \eqref{discrete_1} and \eqref{discrete_4} is well established. Moreover, by Theorem~\ref{thm:generalized_stokesfemwellposed}, the nonconforming linear element method \eqref{discrete_2}-\eqref{discrete_3} is also well-posed. Consequently, the decoupled finite element method \eqref{decoupledFEM} is well-posed. In the following, we present the error analysis for this decoupled method.

We first present error estimates for the Morley-Wang-Xu element method \eqref{discrete_1}.
\begin{lemma}
Let $w \in H_0^2(\Omega) $ be the solution of problem \eqref{continue_1}, and $w_h \in \mathring{W}_h $ be the solution of the discrete problem \eqref{discrete_1}. Assume $ w \in H^3(\Omega)$. We have
\begin{equation}\label{morley_error_1}
|w - w_h |_{2,h} \lesssim h | w |_3 + h^2 \|f \|.
\end{equation}
Furthermore, if the biharmonic equation has $H^3$ regularity, then
\begin{equation}\label{morley_error_2}
\|\nabla_h^2(w - w_h)\|_{-1} \lesssim h^2 | w |_3 + h^3 \|f \|.
\end{equation}
\end{lemma}
\begin{proof}
The estimate \eqref{morley_error_1} was proved in \cite{morley2006}.
We then focus on the proof of estimate \eqref{morley_error_2}.
For any $\boldsymbol{\tau}\in H_0^1(\Omega;\mathbb S)$,
\begin{equation*}
(\nabla_h^2(w - w_h), \boldsymbol{\tau})=-(\nabla_h(w - w_h), \div\boldsymbol{\tau}) + \sum_{T\in\mathcal{T}_h}(\nabla_h(w - w_h), \boldsymbol{\tau}\boldsymbol{n})_{\partial T}.
\end{equation*}
By the weak continuity \eqref{eq:weakcontinuityMWX}, we can derive
\begin{equation*}
\sum_{T\in\mathcal{T}_h}(\nabla_h(w - w_h), \boldsymbol{\tau}\boldsymbol{n})_{\partial T}\lesssim h|w - w_h|_{2,h}|\boldsymbol{\tau}|_1.
\end{equation*}
Consequently,
\begin{equation*}
\|\nabla_h^2(w - w_h)\|_{-1} =\sup_{\boldsymbol{\tau}\in H_0^1(\Omega;\mathbb S)}\frac{(\nabla_h^2(w - w_h), \boldsymbol{\tau})}{\|\boldsymbol{\tau}\|_1}\lesssim |w - w_h|_{1,h} + h|w - w_h|_{2,h}.
\end{equation*}
On the other hand, applying the duality argument as in \cite{Shi1990} yields
\begin{equation}\label{morley_error_12}
| w - w_h |_{1,h} \lesssim h^2 | w |_3 + h^3 \|f \|.
\end{equation}
Therefore, the estimate \eqref{morley_error_2} follows directly from \eqref{morley_error_1} and \eqref{morley_error_12}.
\end{proof}

Next, we establish error estimates for $ \boldsymbol{\sigma}_h$, $\boldsymbol{p}_h$ and $\boldsymbol{r}_h$.
\begin{theorem}\label{lemma_4_2}
Let $\left( w, \boldsymbol{\sigma}, \boldsymbol{p}, 0 \right) \in H_0^2(\Omega) \times H_0^1(\Omega; \mathbb{S} ) \times L^2(\Omega; \mathbb{T}) \times H^1(\Omega; \mathbb{R}^3)$ be the solution of problem \eqref{continue_1}-\eqref{continue_3}, and $ \left( w_h, \boldsymbol{\sigma}_h, \boldsymbol{p}_h, \boldsymbol{r}_h \right) \in \mathring{W}_h \times \mathring{V}_h^{\mathbb{S}} \times \mathbb{P}_0(\mathcal{T}_h; \mathbb{T}) \times V_h $ be the solution of the discrete method \eqref{discrete_1}-\eqref{discrete_3}. Assume that $w\in H^3(\Omega)$, $\boldsymbol{\sigma} \in H^2(\Omega; \mathbb{S})$ and $\boldsymbol{p} \in H^1(\Omega; \mathbb{T})$. Then
\begin{equation}\label{H1bound_sec4}
|\boldsymbol{\sigma} - \boldsymbol{\sigma}_h |_{1,h} + \| \boldsymbol{p} - \boldsymbol{p}_h \|  + \interleave\boldsymbol{r}_h\interleave_{1, h}
\lesssim h(| \boldsymbol{\sigma} |_2+| \boldsymbol{p} |_1 + | w |_3 + h \|f \|).
\end{equation}
Furthermore, if the regularity conditions \eqref{dual_regu}-\eqref{traceless_regu} are satisfied and the biharmonic equation admits $H^3$ regularity, we have
\begin{equation} \label{dual_error_sec4}
\| \boldsymbol{\sigma} - \boldsymbol{\sigma}_h \| \lesssim
h^2 (| \boldsymbol{\sigma} |_2+| \boldsymbol{p} |_1 + | w |_3 + h \|f \|).
\end{equation}
\end{theorem}
\begin{proof}
Estimate \eqref{H1bound_sec4} follows directly from \eqref{H1bound} and \eqref{morley_error_1}, while \eqref{dual_error_sec4} is obtained by combining \eqref{dual_error} with \eqref{morley_error_1}-\eqref{morley_error_2}.
\end{proof}

Finally, we give the error estimates of $ | u - u_h |_{2, h} $ and $| u - u_h |_{1, h}$.
\begin{theorem}\label{theorem_u_h2error}
Let $\left( w, \boldsymbol{\sigma}, \boldsymbol{p}, 0 ,u \right) \in H_0^2(\Omega) \times H_0^1(\Omega; \mathbb{S} ) \times L^2(\Omega; \mathbb{T}) \times H^1(\Omega; \mathbb{R}^3) \times H_0^2(\Omega)$ be the solution of the decoupled formulation \eqref{continue}, and $ \left( w_h, \boldsymbol{\sigma}_h, \boldsymbol{p}_h, \boldsymbol{r}_h , u_h \right) \in \mathring{W}_h \times \mathring{V}_h^{\mathbb{S}} \times \mathbb{P}_0(\mathcal{T}_h; \mathbb{T}) \times V_h\times \mathring{W}_h $ be the solution of the decoupled finite element method \eqref{decoupledFEM}. Assume that $u,w\in H^3(\Omega)$, $\boldsymbol{\sigma} \in H^2(\Omega; \mathbb{S})$ and $\boldsymbol{p}\in H^1(\Omega; \mathbb{T})$. Then
\begin{equation}\label{estimate_uH2}
| u - u_h |_{2, h} \lesssim h ( |\boldsymbol{\sigma}|_2 +  |\boldsymbol{p}|_1 + |u|_3 + |w|_3 + h \|f\|).
\end{equation}
Furthermore, if the regularity conditions \eqref{dual_regu}-\eqref{traceless_regu} are satisfied and the biharmonic equation admits $H^3$ regularity, we have
\begin{equation}\label{estimate_uH1}
| u - u_h |_{1, h}  \lesssim h^2 ( |\boldsymbol{\sigma}|_2 +  |\boldsymbol{p}|_1 + |u|_3 + |w|_3 + h \|f\|).
\end{equation}
\end{theorem}
\begin{proof}
We introduce a new variable $z_h$ to connect equations \eqref{continue_4} and \eqref{discrete_4}. Let $ z_h \in \mathring{W}_h $ satisfy
\begin{equation}\label{H2u1}
\left( \nabla_h^2 z_h, \nabla_h^2 \chi \right) = \left( \boldsymbol{\sigma}, \nabla_h^2 \chi \right), \quad
\forall\,\chi \in \mathring{W}_h.
\end{equation}
Subtracting \eqref{H2u1} from \eqref{discrete_4} gives
\begin{equation*}
\left( \nabla_h^2 \left( u_h - z_h \right), \nabla_h^2 \chi \right) = \left( \boldsymbol{\sigma}_h - \boldsymbol{\sigma}, \nabla_h^2 \chi \right), \quad \forall\,\chi \in \mathring{W}_h.
\end{equation*}
Taking $\chi = u_h - z_h $ in the last equation, we have
\begin{equation*}
| u_h - z_h|_{2,h} \leq \| \boldsymbol{\sigma} - \boldsymbol{\sigma}_h \|.
\end{equation*}
Then using the discrete Poincar\'e inequality \eqref{eq:PoincareCR} and the fact $\nabla_h \mathring{W}_h\subseteq \mathring{V}_h^{\text{CR}}\otimes \mathbb{R}^3$, we get for $j=1,2$ that
\begin{equation}\label{eq:202509040}
| u - u_h |_{j,h} \leq | u - z_h|_{j,h} + |z_h - u_h |_{j,h} \lesssim | u - z_h|_{j,h} + | \boldsymbol{\sigma} - \boldsymbol{\sigma}_h |_{j-1,h}.
\end{equation}

On the other hand, by $\boldsymbol{\sigma}=\nabla^2u$, equation \eqref{H2u1} gives the Galerkin orthogonality
\begin{equation*}
(\nabla_h^2(u-z_h), \nabla_h^2 \chi) = 0, \quad
\forall\,\chi \in \mathring{W}_h,
\end{equation*}
which gives the best-approximation property
\begin{equation*}
|u-z_h|_{2,h}=\inf_{v_h\in\mathring{W}_h}|u-v_h|_{2,h}\lesssim h|u|_3.
\end{equation*}
By a standard duality argument, we also have
\begin{equation*}
|u-z_h|_{1,h}\lesssim h|u-z_h|_{2,h} \lesssim h^2|u|_3.
\end{equation*}
Finally, combining \eqref{eq:202509040}, \eqref{H1bound_sec4}-\eqref{dual_error_sec4}, and the last two inequalities yields \eqref{estimate_uH2}-\eqref{estimate_uH1}.
\end{proof}

\section{Numerical Experiments}\label{section5}
In this section, we present the numerical results of the nonconforming linear element method \eqref{discrete_generalized_stokes} for the generalized tensor-valued Stokes equation~\eqref{continue_generalized_stokes}, and the decoupled finite element method \eqref{decoupledFEM} for the triharmonic equation \eqref{triharmonic_primal}, with all tests conducted on uniform triangulations.

\begin{example}
\rm
Let $\Omega = (0, \pi)^3$. For the generalized tensor-valued Stokes equation \eqref{generalized_stokes}, consider the exact solution
\begin{equation*}
\boldsymbol{\sigma} = \nabla^2(\sin x \sin y \sin z)^3,
\qquad \boldsymbol{r} = 0,
\end{equation*}
\begin{equation*}
\boldsymbol{p} = \curl
\begin{pmatrix}
0 & \sin x \sin y \sin z & 0  \\
\sin x \sin y \sin z & 0 & 0  \\
0 & 0 & 0
\end{pmatrix}.
\end{equation*}
Then $\boldsymbol{\sigma}$ and $\boldsymbol{p}$ satisfy the $\curl$-free and $\div$-free constraints, respectively. From Table~\ref{table_error_stokes}, we observe that numerically
\[
\|\boldsymbol{\sigma} - \boldsymbol{\sigma}_{h}\| = \mathcal{O}(h^2), \quad
|\boldsymbol{\sigma} - \boldsymbol{\sigma}_{h}|_{1, h} = \mathcal{O}(h), \quad
\|\boldsymbol{p} - \boldsymbol{p}_h \| = \mathcal{O}(h),
\]
which are in agreement with the theoretical results \eqref{H1bound} and \eqref{dual_error}.

\begin{table}[htbp]
\setlength{\abovecaptionskip}{0pt}
\setlength{\belowcaptionskip}{0pt}
\caption{ Errors of $ \| \boldsymbol{\sigma} - \boldsymbol{\sigma}_{h} \| $ , $|  \boldsymbol{\sigma} - \boldsymbol{\sigma}_{h} |_{1, h}$ and $ \| \boldsymbol{p} - \boldsymbol{p}_h \| $ }
\label{table_error_stokes}
\begin{tabular}{ccccccc}
\hline
$ h $   \,  &\,   $ \| \boldsymbol{\sigma} - \boldsymbol{\sigma}_{h}\| $ \, & \,   rate    \,  & \,  $| \boldsymbol{\sigma} - \boldsymbol{\sigma}_h |_{1, h}$ \,  &\,  rate   \,  & \,   $\| \boldsymbol{p} - \boldsymbol{p}_h \|$ \,  &\,  rate  \\
\hline
$ 2^{-2}\pi $ \,  &\, 2.297492 \, &  \, -       \, &  \, 14.297881 \,  &\, -   \, & \, 4.810952 \, &\, -   \\
$ 2^{-3}\pi $ \,  &\,  0.677028 \, &  \,  1.7628   \, &  \, 7.623372 \,  &\, 0.9073   \, & \, 2.758073 \, &\,  0.8027  \\
$ 2^{-4}\pi $ \,  &\,  0.181741 \, &  \,  1.8973  \, &  \, 3.872164  \,  &\, 0.9773   \, & \, 1.374580 \, &\, 1.0047 \\
\hline
\end{tabular}
\end{table}
\end{example}

\begin{example}\label{ex1}
\rm
For the triharmonic equation \eqref{triharmonic}, let $\Omega = (0, 1)^3$ and consider the exact solution
\[
u(x, y, z) = \sin^3(\pi x)\sin^3(\pi y)\sin^3(\pi z).
\]

From Tables \ref{table_error_sigma} and \ref{table_error_u}, we observe numerically that
\[
\|\boldsymbol{\sigma} - \boldsymbol{\sigma}_{h}\| = \mathcal{O}(h^2), \,
|\boldsymbol{\sigma} - \boldsymbol{\sigma}_{h}|_{1, h} = \mathcal{O}(h), \,
|u - u_h|_{1, h} = \mathcal{O}(h^2), \,
|u - u_h|_{2, h} = \mathcal{O}(h),
\]
which are consistent with the theoretical results in Theorems \ref{lemma_4_2} and \ref{theorem_u_h2error}.

\begin{table}[htbp]
\setlength{\abovecaptionskip}{0pt}
\setlength{\belowcaptionskip}{0pt}
\caption{ Errors of $ \| \boldsymbol{\sigma} - \boldsymbol{\sigma}_{h} \| $ and $|  \boldsymbol{\sigma} - \boldsymbol{\sigma}_{h} |_{1, h}$}
\label{table_error_sigma}
\begin{center}
\begin{tabular}{ccccccc}
\hline
$ h $   \,  &\,   $ \|\boldsymbol{\sigma} - \boldsymbol{\sigma}_{h} \|$ \, & \,   rate    \,  & \,  $  |\boldsymbol{\sigma} - \boldsymbol{\sigma}_{h} |_{1, h} $ \,  &\,  rate   \\
\hline
$ 2^{-1} $ \,  &\, 11.111193 \, &  \, -       \, &  \, 129.829396\,  &\, -      \\
$ 2^{-2} $ \,  &\, 4.604059 \, &   \, 1.2710  \, &  \, 84.029292 \,  &\, 0.6277 \\
$ 2^{-3} $ \,  &\, 1.332203 \, &   \, 1.7891  \, &  \, 43.453577 \,  &\, 0.9514 \\
$ 2^{-4}$  \,  &\, 0.367327 \, &   \, 1.8587  \, &  \, 21.766158 \,  &\, 0.9974 \\
\hline
\end{tabular}
\end{center}
\end{table}
\begin{table}[htbp]
\setlength{\abovecaptionskip}{0pt}
\setlength{\belowcaptionskip}{0pt}
\caption{ Errors of $ |u-u_{h}|_{1, h}$ and $|u-u_{h}|_{2, h}$}
\label{table_error_u}
\begin{center}
\begin{tabular}{ccccccc}
\hline
$ h $   \,  &\,   $ |u - u_{h}|_{1, h}$ \, & \,   rate    \,  & \,  $| u - u_{h}|_{2, h}$ \,  &\,  rate   \\
\hline
$ 2^{-1} $ \,  &\, 0.682487 \, &  \, -       \, &  \, 10.219628\,  &\, -      \\
$ 2^{-2} $ \,  &\, 0.266772 \, &  \, 1.3552  \, &  \, 6.165173 \,  &\, 0.7291 \\
$ 2^{-3} $ \,  &\, 0.089592 \, &  \, 1.5742  \, &  \, 3.240802 \,  &\, 0.9278 \\
$ 2^{-4} $ \,  &\, 0.027512 \, &  \, 1.7033  \, &  \, 1.624753 \,  &\, 0.9961 \\
\hline
\end{tabular}
\end{center}
\end{table}
\end{example}


%
 \section*{Conflict of interest}

 The authors declare that they have no conflict of interest.

\bibliographystyle{abbrv}
\bibliography{ref}

\end{document}